\documentclass[11pt, reqno, oneside]{amsart}
\usepackage[utf8]{inputenc}
\usepackage[english]{babel}
\usepackage[T1]{fontenc}
\usepackage{amsmath}
\usepackage{amsfonts}
\usepackage{amssymb}
\usepackage{amsthm}
\usepackage{amscd}
\usepackage{soul}
\usepackage{geometry}
\usepackage{enumitem} 
\usepackage{cancel} 
\usepackage{graphicx}
\usepackage{pgf}

\usepackage{mathtools}
\usepackage{color}
\usepackage{comment}
\usepackage{mathdots}
\usepackage{algorithm}
\usepackage{algpseudocode}
\usepackage{tikz}
\usetikzlibrary{arrows}
\usepackage[notcite,notref]{}
\usepackage{color}
\usepackage{bm}
\usepackage[normalem]{ulem} 
\definecolor{marin}{rgb}   {0.,   0.3,   0.7} 
\definecolor{rouge}{rgb}   {0.8,   0.,   0.} 
\definecolor{sepia}{rgb}   {0.8,   0.5,   0.} 
\definecolor{uuuuuu}{rgb}{0.26666666666666666,0.26666666666666666,0.26666666666666666}
\usepackage[colorlinks,citecolor=marin,linkcolor=rouge,
            bookmarksopen,
            bookmarksnumbered
           ]{hyperref}
\usepackage[noabbrev]{cleveref}
\usepackage[multiuser,draft,inline,marginclue]{fixme}
\fxusetheme{colorsig}
\FXRegisterAuthor{rou}{arou}{Robin}
\FXRegisterAuthor{cha}{acha}{Quentin}

\usetikzlibrary{calc,backgrounds,fit}
\usetikzlibrary{external}
\usetikzlibrary{patterns}
\usetikzlibrary{decorations.pathreplacing,decorations.markings,arrows}
\usetikzlibrary{shapes,intersections,positioning}
\usetikzlibrary{arrows}

\newcommand\N{\mathbb{N}}
\newcommand\T{\mathbb{T}}
\newcommand\Z{\mathbb{Z}}

\newcommand\R{\mathbb{R}}
\newcommand\C{\mathbb{C}}

\newcommand{\dd}{\mathrm{d}}

\newcommand{\enstq}[2]{\left\{#1~\middle|~#2\right\}}

\renewcommand{\Re}{\operatorname{Re}}

\newcommand\Id{\mathrm{Id}}

\newcommand\argmin{\mathrm{argmin}}

\newcommand\diam{\mathrm{diam}}
\newcommand\centre[4][below]{\node (#3) at #2 [circle,minimum size=0.5em,inner sep=0pt,thin,fill,solid] {}; \node [#1=0.1em] at (#3) {#4};}

\newcommand{\uu}{u^{*}}
\newcommand{\uphi}{\phi^{*}}

\theoremstyle{plain}
\newtheorem{theorem}{Theorem} [section]

\newtheorem{assumption}[theorem]{Assumption}
\newtheorem{lemma}[theorem]{Lemma}

\newtheorem{proposition}[theorem]{Proposition}

\theoremstyle{remark}
\newtheorem{remark}[theorem]{Remark}

\numberwithin{equation}{section}

\definecolor{alizarin}{rgb}{0.82, 0.1, 0.26}

\title[FV minimizers and normalized gradient flow for the rotating GPE]{Finite volume discretization and normalized gradient flow for computing minimizers of the rotating Gross--Pitaevskii energy}

\author{Quentin Chauleur}
\address{Univ. Lille, CNRS, Inria, UMR 8524 - Laboratoire Paul Painlevé, F-59000 Lille, France. }
\email{quentin.chauleur@inria.fr}

\author{Guillaume Dujardin}
\address{Univ. Lille, CNRS, Inria, UMR 8524 - Laboratoire Paul Painlevé, F-59000 Lille, France. }
\email{guillaume.dujardin@inria.fr}

\author{Robin Roussel}
\address{Univ. Lille, CNRS, Inria, UMR 8524 - Laboratoire Paul Painlevé, F-59000 Lille, France. }
\email{robin.a.roussel@inria.fr}

\begin{document}

\begin{abstract}
We study the approximation by a finite volume scheme of the stationary Gross--Pitaevskii equation with angular momentum rotation. We prove that, under standard mesh regularity assumptions, the discrete global minimizers converge to the continuous one as the mesh diameter tends to zero. Moreover, under a spectral gap assumption on the Hessian of the discrete energy, we prove that solutions to a normalized continuous gradient flow converges exponentially fast to such discrete minimizers. Additionally, we show that this previous discrete spectral gap assumption can be derived from similar hypotheses in the continuous setting. We support our theoretical results with numerical simulations, which broaden the scope of the analysis.
\end{abstract}

\maketitle

%
%

\section{Introduction}

Let $\mathcal{D} \subset \R^d$ be an open bounded polyhedral subset (Lipschitz domain) of $\R^d$ for $d \in \left\{2,3 \right\}$ with boundary $\partial \mathcal{D}$. We consider the Gross--Pitaevskii equation with a magnetic rotation term
\begin{equation} \label{eq:GP} \tag{GP}
 i \partial_t \psi + \Delta \psi = |\psi|^2 \psi + V \psi - \Omega L \psi  
\end{equation}   
and homogeneous Dirichlet boundary conditions $\psi(x)=0$ for all $x \in \partial \mathcal{D}$. The function $V:\mathcal{D} \to \R_+$ denotes a trapping potential. The operator $L=L_{3}=-i (x_1 \partial_{x_2} -x_2 \partial_{x_1})$ denotes the $x_3$-component of the angular momentum operator $\mathcal{L}=-i x  \wedge \nabla$, with $x =(x_1,x_2)$ if $d=2$ or $x=(x_1,x_2,x_3)$ if $d=3$. Finally, the constant $\Omega \in \R$ denotes the angular velocity.

This equation serves as a fundamental model to describe the dynamics of Bose-Einstein condensates, a state of matter near absolute zero temperature whose evolution can be characterized by a single wave function $\psi$. When set to rotation in a strong confinement regime, such condensates exhibit quantized vortex nucleation processes. This phenomena has been the subject of intense research from the scientific community over the past decades.

Solutions of equation \eqref{eq:GP} enjoy both the conservation of mass $\| \psi(t)\|_{L^2}=\| \psi(0)\|_{L^2} $ for all time $t\in \R$ and the conservation of the energy 
\begin{equation}\label{eq:GP_energy} \tag{GPE}
    E(\phi)= \int_{\mathcal{D}} \left(  \frac{1}{2}|\nabla \phi |^2 + \frac{1}{2}V |\phi|^2  - \frac{1}{2}\Omega \overline{\phi} L \phi + \frac{1}{4} |\phi|^4  \right).
\end{equation}
In this work, we are interested in computing the \textit{ground state} solution of \eqref{eq:GP} at a given $\Omega \in \R$, which minimizes the energy under a mass normalization constraint. As uniqueness may only be provided up to usual invariants of \eqref{eq:GP_energy}, we rather consider the \textit{ground state manifold} 
  \[ \mathcal{U} = \mathop{\argmin}_{\phi \in \mathbb{S}} E(\phi),  \]
where $\mathbb{S} = \{\phi \in H^1_0(\mathcal{D}) \mid \|\phi\|_{L^2}=1\}$. In the (fully real-valued) case without rotation $\Omega=0$, since the seminal work of \cite{BaoDu2004}, a number of numerical methods have been analyzed to compute ground states of \eqref{eq:GP}, and we refer to the reviews \cite{BaoCai2013,HenningJarlebring2025} and references within. While a strong focus has been put on the study of time-iterative solvers, far less works deal with the analysis of minimizers for spatial discretization of \eqref{eq:GP}. One should mention the early contribution from \cite{Zhou2004} for general discretizations, followed by the work of \cite{CancesChakirMaday2010} on Fourier spectral and finite elements approximations. Generalized finite element schemes were further developed in \cite{HenningMalqvistPeterseim2014,HenningPersson2023}.

Adding rotation $\Omega \neq 0$ makes the problem at hand significantly more difficult: the ground state may be complex-valued, is invariant by multiplication with a complex phase shift $e^{i\theta}$ for $\theta \in \R$, and by rotation if the potential $V$ and the domain $\mathcal{D}$ are radially symmetric. The degeneracy of such symmetry-related minimizers thus complicates theoretical studies while potentially slowing down numerical simulations. We refer to \cite{FengHenningTang2026} for a description of the so-called quasi-isolated ground states of \eqref{eq:GP_energy}, an assumption that, at each minimizer, the energy has a strict local minimum outside the directions corresponding to these invariants. 

To the best of the authors' knowledge, the only works that take into account both spatial discretization and angular momentum rotation for \eqref{eq:GP_energy} are \cite{HenningYadavDiscrete2025,HenningPersson2026}, where error analysis is performed in a finite element framework. The aim of this paper is to extend such analysis to a finite volume discretization of \eqref{eq:GP_energy}, following earlier contributions from the authors \cite{Chauleur2026,ChauleurChicireanu2026} and avoiding the standard variational formulation of finite elements. In particular, we propose for the first time in the literature a direct structure-preserving discretization of the angular momentum operator $L$ on non-cartesian meshes. Moreover, finite volume schemes have proven to be useful for simulating classical turbulent flows \cite{Feireisl2019} and efficiently capturing vortex nucleation processes \cite{Busto2023}. We continue our study on the adaptation of such techniques for quantum fluids (see also \cite{CalgaroColombierCreuse2026} for a recent contribution), since main quantities of interest such as vorticity \cite{brachet2023,krstulovic2016} are deeply related to fluids mechanics, see e.g. \cite{Aftalion2006} for a rigorous mathematical study of quantum vortices. We furthermore remove the usual convex assumption on $\mathcal{D}$ of previous works by only considering open bounded polyhedral subset of $\R^d$, which include recent experimental setup of physical relevance such as rings geometries \cite{ChauleurChicireanu2026} or shell-shaped condensates \cite{Rhyno2026}. 

This paper is organized as follows. In Section~\ref{sec:framework_and_results} we introduce standard notations for finite volume schemes as well as a finite volume discretization of \eqref{eq:GP_energy} before stating our main results: Theorem~\ref{th:cv_ground_states} and Theorem~~\ref{th:exp_convergence}. In Section~\ref{sec:discrete_function_analysis} we both recall and develop some discrete functional analysis tools required for our study, in particular for the discretization of the operator $L$. Section~\ref{sec:minimizers} is then dedicated to the proof of Theorem~\ref{th:cv_ground_states}, showing that both the discrete energy and minimizers converge towards their continuous counterparts. We then prove in Section~\ref{sec:gradient_flow} that the standard normalized gradient flow from~\cite{BaoDu2004,FaouJezequel2018} exponentially converges to the discrete ground state manifold in our rotational finite volume setting, with an exponential rate independent of the mesh size under a quasi-isolation hypothesis on the continuous ground-state. Finally, we corroborate our theoretical results by numerical simulations in Section~\ref{sec:numerics} (for $d=2$), providing further discussions and open perspectives.

To avoid confusion, continuous (\textit{resp.} discrete) functions are denoted by greek (\textit{resp.} latin) letters and minimizers are indicated with a star as superscript.


\section{Discrete framework and main results} \label{sec:framework_and_results}

\subsection{Finite volume meshes}

We consider a mesh $\mathcal{M} = (\mathcal{F}, \mathcal{E})$ on $\mathcal{D}$, where $\mathcal{F}$ is a set of disjoint convex open polyhedra such that $\cup_{K \in \mathcal{F}} \bar{K} = \bar{\mathcal{D}}$ and $\mathcal{E}$ is a set of disjoint open subsets of hyperplanes such that $\cup_{K \in \mathcal{F}} \partial K = \cup_{\sigma \in \mathcal{E}} \bar{\sigma}$. For all $K$ and $L$ in $\mathcal{F}$, $\bar{K} \cap \bar{L}$ either has zero $(d-1)$-dimensional Lebesgue measure, or $\bar{K} \cap \bar{L} = \bar{\sigma} \in \mathcal{E}$. In this case, we denote $\sigma = K|L$ and say that $\sigma \in \mathcal{E}_{\mathrm{int}}$. We conversely define the set of exterior edges $\mathcal{E}_{\mathrm{ext}} = \mathcal{E} \backslash \mathcal{E}_{\mathrm{int}}$. Moreover, for $K \in \mathcal{F}$, we define $\mathcal{E}_K = \{\sigma \in \mathcal{E} \mid \sigma \subset \overline{K}\}$.

Throughout the paper, all meshes are assumed to be \textit{admissible} in the following sense: for all $K$ in $\mathcal{F}$, there exists a point $x_K \in K$, such that if $\sigma = K|L$, then the segment $[x_K, x_L]$ is orthogonal to the hyperplane containing $\sigma$. For all $K$ in $\mathcal{F}$ and $\sigma \in \mathcal{E}$, we also define $|K|$ as the $d$-dimensional Lebesgue measure of $K$, $|\sigma|$ as the $(d-1)$-dimensional Lebesgue measure of $\sigma$, $x_\sigma$ as the center of mass of $\sigma$ and $n_{\sigma, K}$ as the unit vector normal to $\sigma$ and pointing away from $K$ if $\sigma \in \mathcal{E}_K$. Also, for $K \in \mathcal{F}$ and $\sigma \in \mathcal{E}_K$, we define $d_{\sigma, K} = \inf_{x \in \sigma}|x_K - x|$ and 
$$
d_\sigma \coloneqq 
\begin{cases}
    d_{\sigma, K} + d_{\sigma, L} &\text{if } \sigma=K|L,
    \\
    d_{\sigma, K} &\text{if } \sigma \in \mathcal{E}_{\mathrm{ext}}.
\end{cases}
$$
Finally, for $K \in \mathcal{F}$, we denote by $\diam(K)$ its diameter and define $h = \max_{K \in \mathcal{F}} \diam(K)$. As is often the case, we will regularly use the abuse of notation of referring to the mesh dependency of the objects using only the mesh diameter $h$ rather than the complete mesh~$\mathcal{M}$.

We introduce the following measure of regularity for the mesh $\mathcal{M} = (\mathcal{F}, \mathcal{E})$:
\begin{equation} \label{eq:mesh_regularity}
\xi=\xi_h \coloneqq \min \left( \min_{K \in \mathcal{F},\sigma \in \mathcal{E}_K}\frac{d_{\sigma,K}}{\diam(K)},\min_{K \in \mathcal{F},\sigma \in \mathcal{E}_K} \frac{d_{\sigma,K}}{d_\sigma}  \right) \leq 1,
\end{equation}  
 which is standard to obtain convergence results (see for example \cite{CoudiereDiscrete2001,EymardGallouetHerbin2000}). For $\Xi \in (0, 1]$, we say that an admissible mesh $\mathcal{M}$ is $\Xi$-regular if $\xi_h \geq \Xi$. Moreover, a sequence of admissible meshes $(\mathcal{M}_n)_{n \in \N}$ is said to be \textit{uniformly regular} if there exists a constant $\Xi \in (0, 1]$ such that $\mathcal{M}_n$ is $\Xi$-regular for all $n \in \N$. Note that without explicit mention in the results of this article, this regularity property is not assumed.


For a given mesh $\mathcal{M}$, we now define the associated \textit{diamond mesh} $\mathcal{M}^*=(\mathcal{F}^*,\mathcal{E}^*)$. Using similar notations as in \cite[Section 9]{EymardGallouetHerbin2000}, for $K \in \mathcal{F}$ and $\sigma \in \mathcal{E}_K$, we define 
$$
\mathcal{V}_{\sigma, K} \coloneqq \{tx + (1-t)x_K, x \in \sigma, t \in (0, 1)\},
$$
as well as 
$$
\mathcal{V}_\sigma \coloneqq 
\begin{cases}
    \mathrm{int}\left(\overline{\mathcal{V}_{\sigma, K}} \cup \overline{\mathcal{V}_{\sigma, L}}\right) &\text{if } \sigma = K|L, 
    \\
    \mathcal{V}_{\sigma, K} &\text{if } \sigma \in \mathcal{E}_{\mathrm{ext}}.
\end{cases}
$$
We then define the set of diamond faces $\mathcal{F}^* \coloneqq \enstq{\mathcal{V}_\sigma}{\sigma \in \mathcal{E}}$, and we naturally define the associated set of diamond edges $\mathcal{E}^*$, see e.g. \cite{EymardGallouetHerbin2000}. Note that $|\mathcal{V}_{\sigma, K}| = d_{\sigma, K}|\sigma|/d$ and $|\mathcal{V}_\sigma| = d_\sigma|\sigma|/d$. We refer to Figure \ref{fig:notation_mesh} for an illustration of these notations.

\begin{figure}[htbp!]
\centering
\begin{tikzpicture}[scale=2]

  \clip (-1.2,-0.6) rectangle (1.8,1.3);

  \node[rectangle,fill] (a) at (-1,0.6) {};
  \node[rectangle,fill] (b) at (0,1.2) {};
  \node[rectangle,fill] (c) at (0,-0.2) {};
  \node[rectangle,fill] (d) at (1.5,0.3) {};

  \centre[above right]{(-0.6,0.5)}{xK}{$x_K$};
  \centre[above left]{(0.9,0.5)}{xL}{$x_L$};
  
  \draw[thick] (b)--(c) node [pos=0.7,right] {$\sigma=$\small{$K|L$}};

  \draw[thin,opacity=0.5] (a) -- (b) -- (c) -- (a) ;
  \draw[thin,opacity=0.5] (d) -- (b) -- (c) -- (d);

  \draw[dashed] (xK) -- (xL);
  
  \coordinate (KK) at ($(xK)!0.65!-90:(xL)$);
  \coordinate (LL) at ($(xL)!0.65!90:(xK)$);

  \draw[dotted,thin] (xK) -- (KK);
  \draw[dotted,thin] (xL) -- (LL);

  \draw[|<->|] (KK) -- (LL) node [midway,fill=white,sloped] {$d_{\sigma}$};
 
  \coordinate (KAB) at ($(a)!(xK)!(b)$);
  \coordinate (KAC) at ($(a)!(xK)!(c)$);

  \coordinate (LDB) at ($(d)!(xL)!(b)$);
  \coordinate (LDC) at ($(d)!(xL)!(c)$);

  \draw[dashed] (xK) -- ($(xK)!3!(KAB)$);
  \draw[dashed] (xK) -- ($(xK)!3!(KAC)$);

  \draw[dashed] (xL) -- ($(xL)!3!(LDB)$);
  \draw[dashed] (xL) -- ($(xL)!3!(LDC)$);

  \begin{scope}[on background layer]   
    \draw (0,0.5) rectangle ++ (0.1,-0.1);
  \end{scope}
  
  \coordinate (nukl) at ($(b)!0.3!(c)$);
  \draw[->,>=latex] (nukl) -- ($(nukl)!0.3cm!90:(c)$) node[above] {$n_{\sigma, K}$};

\end{tikzpicture}
\qquad
\begin{tikzpicture}[scale=2]

  \clip (-1.2,-0.6) rectangle (1.8,1.6);

  \coordinate (xK) at (-0.5,0.5);
  \coordinate (xL) at (0.8,0.5);

  \node[rectangle,fill] (a) at (-1,0.6) {};
  \node[rectangle,fill] (b) at (0,1.2) {};
  \node[rectangle,fill] (c) at (0,-0.2) {};
  \node[rectangle,fill] (d) at (1.5,0.3) {};

  \centre[above left]{(-0.6,0.5)}{xK}{$x_K$};
  \centre[below right]{(0.9,0.5)}{xL}{$x_L$};

  \coordinate (KK) at ($(xK)!0.65!-90:(xL)$);
  \coordinate (LL) at ($(xL)!0.65!90:(xK)$);

  \draw[thin,opacity=0.5] (a) -- (b) -- (c) -- (a) ;
  \draw[thin,opacity=0.5] (d) -- (b) -- (c) -- (d);
  \draw[dashdotted] (xK) -- (c) -- (xL) -- (b) -- (xK);

  \draw[thick] (b) -- (c);

  \draw[thick,dashed] (xK) -- (xL);

  \draw (0,0.5) node [above right]  {$\mathcal{V}_\sigma$};
  \draw (0,0.3) node [left]  {$\sigma$};

  \begin{scope}[on background layer]
    \fill [diamond,fill opacity=0.1] (xK.center) -- (c.center)-- (xL.center) -- (b.center);
    \draw (0,0.5) rectangle ++ (0.1,-0.1);
  \end{scope}

\end{tikzpicture}
\caption{Notations for the mesh $\mathcal{M}$ and the diamond mesh $\mathcal{M}^*$ for $\sigma \in \mathcal{E}_{\mathrm{int}}$.}
\label{fig:notation_mesh}
\end{figure}
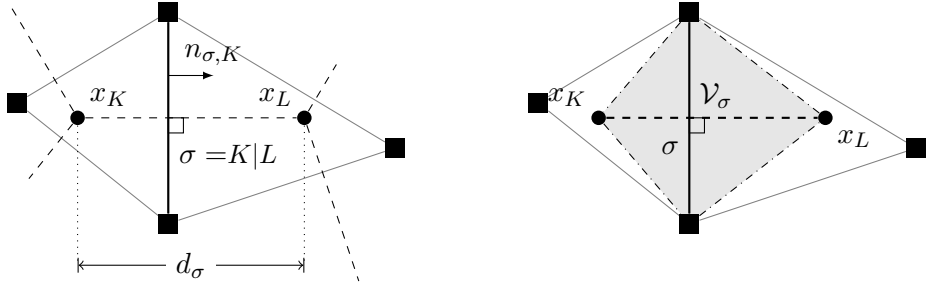

We now turn to the definitions of the discrete functional spaces. We define $X(\mathcal{M})$ as the space of complex-valued functions which are piecewise constant on the cells $\mathcal{F}$. Similarly, we define $X(\mathcal{M}^*)$ as the space of complex-valued functions which are piecewise constant on the diamond cells $\mathcal{F}^*$. We note that these spaces inject into $L^p(\mathcal{D})$ for all $p \geq 1$ and that, for $u \in X(\mathcal{M})$ and $v \in X(\mathcal{M}^*)$, we have 
\[    \|u\|_{L^p}^p = \sum_{K \in \mathcal{F}} |u_K|^p |K| \quad \text{and} \quad \|v\|_{L^p}^p = \sum_{\sigma \in \mathcal{E}} |v_\sigma|^p |\mathcal{V}_\sigma|= \sum_{\sigma \in \mathcal{E}} |v_\sigma|^p \frac{d_\sigma |\sigma|}{d}. \]
We recall the definition of discrete Sobolev spaces. First, we define for $K \in \mathcal{F}$ and $\sigma \in \mathcal{E}_K$
$$
D_{\sigma, K} u \coloneqq 
\begin{cases}
    u_L - u_K &\text{if } \sigma=K|L,
    \\
    -u_K &\text{if } \sigma \in \mathcal{E}_{\mathrm{ext}}.
\end{cases}
$$
We note that $|D_{\sigma, K} u| = |D_{\sigma, L} u|$, so that we may write $|D_\sigma u|$ without ambiguity. Similarly, $(D_{\sigma, K} u) n_{\sigma, K} = (D_{\sigma, L} u) n_{\sigma, L}$ so that we may write $(D_\sigma u) n_\sigma$. We then define the norm
$$
|u|_{1, p}^p \coloneqq \sum_{\sigma \in \mathcal{E}} \frac{|\sigma|}{d_\sigma^{p-1}} |D_\sigma u|^p.
$$
for any $p \geq 1$ as the discrete counterpart of standard homogeneous seminorms on the Sobolev spaces $W^{1,p}_0(\mathcal{D})$ or $H^1_0(\mathcal{D})$ for $p=2$. Finally, we define the discrete interpolation function $I_h : X(\mathcal{M}) \rightarrow X(\mathcal{M}^*)$ as 
$$
(I_h u)_\sigma \coloneqq 
\begin{cases}
    \frac{1}{2}(u_K + u_L) &\text{if } \sigma= K|L, 
    \\ 
    \frac{1}{2}u_K &\text{if } \sigma \in \mathcal{E}_{\mathrm{ext}}.
\end{cases}
$$
For conciseness, we will generally write $I_\sigma u := (I_h u)_\sigma$. We also define $I_h^* : X(\mathcal{M}^*) \rightarrow X(\mathcal{M})$ as the adjoint of $I_h$ with respect to the standard $L^2$-product. 

Throughout the article, we will be using two notions of inner-products. For $f, g \in L^2(\mathcal{D})$, we define 
$$
(f, g) = \int_\mathcal{D} f\overline{g} \text{ and } \langle f, g \rangle = \Re(f, g),
$$
hence $(\cdot, \cdot)$ corresponds to the usual Hermitian product on the complex Hilbert space $L^2(\mathcal{D})$, while $\langle \cdot, \cdot\rangle$ is the corresponding real inner product. Indeed, since the energy is always real-valued, it is not complex Fréchet differentiable. When considering differentials of the energy, we therefore need to endow $L^2(\mathcal{D})$ with a Hilbert structure over $\R$, as usually done for nonlinear Schrödinger type equations, see e.g. \cite{Cazenave}.

\subsection{Finite volume discretization}

We now introduce the finite volume discretization of the Gross--Pitaesvkii energy \eqref{eq:GP_energy} that will be studied throughout the paper. We begin by recalling the definition of the two points flux approximation (TPFA) of the Dirichlet Laplacian. For $u \in X(\mathcal{M})$ and $K \in \mathcal{F}$, we define 
$$
(A_h u)_K \coloneqq \frac{1}{|K|}\sum_{\sigma \in \mathcal{E}_K} \frac{|\sigma|}{d_\sigma} D_{\sigma, K}u.
$$
Note that $|u|_{1, 2}^2 = -( A_h u, u )$. We then introduce a discretization of the angular momentum operator $L$. We recall that for any $\varphi \in \mathcal{C}^1(\bar{\mathcal{D}})$, we have $L \varphi = i R \cdot \nabla \varphi$, where $R$ is given by 
$$
R(x) = 
\begin{cases}
    (x_2, -x_1) &\text{if } d=2,
    \\
    (x_2, -x_1, 0) &\text{if } d=3.
\end{cases}
$$
Since $R$ is divergence free, we also have $L \varphi = i \mathrm{div}(R \varphi)$. From this, we may write 
$$
\int_{K} L \varphi = i\sum_{\sigma \in \mathcal{E}_K} \int_\sigma \varphi(x) (R(x) \cdot n_{\sigma, K}).
$$
This formula motivates the definition of the discretized angular momentum
\begin{equation} \label{eq:def_Lz}
(L_h u)_K \coloneqq \frac{i}{|K|} \sum_{\sigma \in \mathcal{E}_K} |\sigma|(I_\sigma u)(R(x_\sigma) \cdot n_{\sigma, K}), 
\end{equation}    
for $u$ in $X(\mathcal{M})$ and $K \in \mathcal{F}$. For the potential part, we consider $V\in W^{1,\infty}(\mathcal{D})$  throughout this paper, and we assume that there exists a constant $\kappa>0$ such that for all $x \in \mathcal{D}$,
\begin{equation} \label{eq:assumption_V}
V(x) \geq (1+\kappa)\frac{d \Omega^2}{4}|x|^2.
\end{equation}    
We then set 
\[    V_h u \coloneqq I_h^* \left[\left(P_h^\diamond V\right) I_h u\right] \]
for the discretization of the potential $V$, where the projection $P_h^\diamond V \in X(\mathcal{M}^*)$ is given by $(P_h^\diamond V)_\sigma := V(x_\sigma)$. As will be seen in Section \ref{sec:minimizers}, this discretization of the potential energy on the diamond mesh will prove to have good properties when considering the coercivity of the discretized energy. We note in passing that, for $u\in X(\mathcal{M})$, we have 
\[    \left( V_h u, u \right) = \left( P_h^\diamond V I_h u, I_h u \right) = \sum_{\sigma \in \mathcal{E}} V(x_\sigma) |I_\sigma u|^2 |\mathcal{V}_\sigma|. \]
We finally define the discrete energy functional $E_h : X(\mathcal{M}) \rightarrow \R$ as
\begin{equation}\label{eq:discrete_GP_energy} \tag{DGPE}
    E_h(u) \coloneqq \frac{1}{2}\left( H_h u, u \right) + \frac{1}{4}\|u\|_{L^4}^4,
\end{equation}
where the discrete linear operator $H_h$ is given by 
$$
H_h = -A_h + V_h - \Omega L_h.
$$
Note that one needs to prove that $H_h$ is self-adjoint in order to show that $E_h$ is real-valued, which will be done in Section \ref{sec:minimizers}. We also note that the gradient of the energy is given by
\begin{equation}\label{eq:energy_gradient}
E_h'(u) = H_h u + |u|^2 u,  \quad u \in X(\mathcal{M})
\end{equation}
and that for any $u \in X(\mathcal{M})$, the Hessian $E''_h(u) : X(\mathcal{M}) \rightarrow X(\mathcal{M})$ is given by
\begin{equation}\label{eq:energy_hessian}
    E_h''(u) v = H_h v + |u|^2v + 2 u \mathrm{Re}(\overline{u}v), \quad v \in X(\mathcal{M}).
\end{equation}


\subsection{Main results}
\label{subsec:mainresults}

We begin by considering the minimization problem of the discrete energy under normalized mass constraint. For an admissible mesh $\mathcal{M}$, we define the set
\begin{equation} \label{eq:discrete_minimization_pb}
\mathcal{U}_h = \mathop{\argmin}_{u \in \mathbb{S}_h} E_h(u),
\end{equation}
where $\mathbb{S}_h \coloneqq \enstq{u \in X(\mathcal{M})}{ \|u\|_{L^2}=1}$. Since $X(\mathcal{M})$ is a finite dimensional space, it is clear that $\mathcal{U}_h$ is not empty. However, in the continuous case, it is known that minimizers exist only when the angular velocity $\Omega$ is small enough not to overcome the confining potential~$V$, see e.g. \cite[Theorem 5.2.]{BaoCai2013}. In our case, the lower estimate \eqref{eq:assumption_V} guarantees the existence of a ground state in the continuous setting. Compared to the classical assumption on the confining potential for the existence of ground states \cite{HenningYadavDiscrete2025}, we note the presence a factor depending on the dimension in \eqref{eq:assumption_V}. As it will be clear in the computations of Section~\ref{sec:minimizers}, this factor is needed for proving coercivity of the discrete energy. Note that this is essentially harmless for the case of strong confining potentials such as quartic potentials, as well as for the usual case of a harmonic potential in low rotation regimes. Numerical investigations tend to show that it may be an artifact of proof, so that the additional factor~$d$ can be removed in practice.

Our first result concerning the convergence of discrete minimizers is the following.
\begin{theorem}\label{th:cv_ground_states}
    Let $(\mathcal{M}_n)_{n \in \N}$ be a sequence of uniformly regular admissible meshes with $h_n := h_{\mathcal{M}_n} \rightarrow 0$, and $\uu_n$ be a minimizer in $\mathcal{U}_{h_n}$. Then up to a subsequence, $(\uu_n)_{n \in \N}$ converges in $L^2(\mathcal{D})$ to a continuous minimizer $\uphi \in \mathcal{U}$. Furthermore, $(E_{h_n}(\uu_n))_{n\in \N}$ converges to the ground state energy $E(\uphi)$.
\end{theorem}

We now introduce a discrete $L^2$-normalized gradient flow to compute such discrete minimizers. For $u(0) \in \mathbb{S}_h$, we consider 
\begin{equation}\label{eq:gradient_flow} \tag{GF}
    \partial_t u(t) = - E_h'(u(t)) + \mu_h(u(t))u(t),
\end{equation}
for all $t \geq 0$, where $ E_h'(u)$ is given by \eqref{eq:energy_gradient} and 
\begin{equation}\label{eq:discrete_chemical_potential}
    \mu_h(u) = \frac{1}{\|u\|_{L^2}^2}\langle  E_h' (u), u\rangle
\end{equation}
is a Lagrange multiplier ensuring that the solutions stay normalized over time, since
\[ \frac12 \frac{\dd}{\dd t} \|u\|_{L^2}^2 =  \langle \partial_t u, u\rangle=  -\langle E_h'(u), u \rangle + \frac{1}{\|u\|_{L^2}^2}\left\langle \langle E_h'(u), u \rangle u, u\right\rangle =0 . \]
We also note that 
\begin{equation}\label{eq:alternative_formula_mu}
  \mu_h(u) = \frac{1}{\|u\|_{L^2}^2}\left(2E_h(u) + \frac{1}{2}\|u\|_{L^4}^4\right).
\end{equation}
Furthermore, it is straightforward to see that any solution $u$ of \eqref{eq:gradient_flow} is energy diminishing:
\[ \frac{\dd}{\dd t} E_h(u) = \langle E_h'(u), \partial_t u \rangle= \langle -\partial_t u + \mu_h(u) u, \partial_t u\rangle= -\|\partial_t u\|^2 + \frac12 \mu_h(u) \frac{\dd}{\dd t}\|u\|^2 = -\|\partial_t u\|^2 \leq 0.  \]
Our second main result shows the local exponential convergence of solutions to \eqref{eq:gradient_flow} to (local) minimizers. We start by defining the set of critical points of the discrete energy on the unit $L^2$-sphere:
\begin{equation}
  \label{eq:crit} \mathrm{Crit}_h \coloneqq \enstq{u \in \mathbb{S}_h}{ E_h'(u) = \mu_h(u)u}.
\end{equation}
As \eqref{eq:discrete_minimization_pb} stands as a minimization problem on the sphere $\mathbb{S}_h$, we consider the tangent space to $\mathbb{S}_h$ which is given by the following orthogonal space with respect to $u \in \mathbb{S}_h$:
\[ T_u \mathbb{S}_h = \enstq{v \in X(\mathcal{M})}{ \langle u, v\rangle=0}.\]
We also write $T_u^\C\mathbb{S}_h = T_u\mathbb{S}_h \cap T_{iu}\mathbb{S}_h$, which denotes the orthogonal complement of $u$ with respect to the complex $L^2$-inner product $(\cdot,\cdot)$.

Symmetries of the energy $E_h$, such as invariance by a phase shift $e^{i\theta}u$ in our discrete case, leads to degenerate directions in the tangent space that should be removed in order to infer local coercivity. In fact, differentiating twice $E_h(e^{i\theta}u)$ with respect to $\theta$ shows that the operator $E_h''(u) - \mu_h(u)$ is vanishing on the direction of $iu$. We then introduce a quasi-coercivity condition (see also \cite[Definition 2.2]{HenningYadavDiscrete2025}) on the Hessian of the energy.

\begin{assumption}\label{ass:quasi_coercivity}
    A critical point $\uu \in \mathrm{Crit}_h$ is said to be quasi-isolated if there exists a constant $C > 0$ such that, for all $v \in T_{\uu}^\C\mathbb{S}_h$,
\[    \left\langle (E_h''(\uu) - \mu_h(\uu)\Id)v, v\right\rangle \geq C|v|_{1, 2}^2. \]
\end{assumption}

Note that an analogous assumption in the continuous case, known as the \textit{fundamental gap conjecture} \cite[Conjecture 1.6]{AntonelliShakarov2024}, is considered to be true for the ground state of \eqref{eq:GP_energy}, as numerical simulations and formal expansions confirm this conjecture \cite{BaoRuan2018}, although a rigorous proof is still open. Moreover, we will prove in Section~\ref{sec:spectral_gap} using a spectral convergence argument that the existence of a spectral gap in the continuous case implies Assumption~\ref{ass:quasi_coercivity} on meshes with small enough diameter.

Our second main result on local convergence of the gradient flow \eqref{eq:gradient_flow} reads as follows.
\begin{theorem}\label{th:exp_convergence}
    Suppose $\uu \in \mathrm{Crit}_h$ satisfies Assumption \ref{ass:quasi_coercivity}. Then, there exists $C>0$ such that, for all $\varepsilon > 0$, there exists $\delta > 0$ such that, if $\inf_{\theta \in \R}|u(0)-e^{i\theta}\uu|_{1,2} < \delta$, then for all $t \geq 0$,
    $$
    \inf_{\theta \in \R} \|u(t) - e^{i\theta}\uu\|_{L^2} \leq C e^{-(\lambda-\varepsilon)t} \inf_{\theta \in \R} \|u(0) - e^{i\theta}\uu\|_{L^2},
    $$
    where $u$ denotes the solution to \eqref{eq:gradient_flow} and the discrete spectral gap $\lambda$ is given by
\begin{equation} \label{eq:lambda}
\lambda \coloneqq \inf_{v \in T_{\uu}^\C \mathbb{S}_h\backslash\{0\}} \frac{\left\langle \left(E_h''(\uu) - \mu_h(\uu) \Id \right)v, v\right\rangle}{\|v\|_{L^2}^2} >0.
\end{equation}     
\end{theorem}

Note that $\lambda$ depends implicitly on the mesh, implying in practice that the convergence rate may become degenerate in the limit $h \to 0$. However, we will see in Section~\ref{sec:spectral_gap} that, if $(u_n^*)_{n \in \N}$ is a sequence of discrete minimizers converging to a continuous minimizer $\phi^*$ as in Theorem~\ref{th:cv_ground_states}, the associated discrete convergence rates $\lambda_n$ converge to the continuous spectral gap $\lambda$. In particular, if the limiting ground-state is quasi-isolated, the discrete convergence rates may be bounded uniformly in $n$. However, if the limiting $\lambda$ is equal to zero, the convergence becomes degenerate in the limit. This may in particular be the case for axisymmetric potentials, as will be illustrated in Section~\ref{sec:numerics}.

To keep the presentation clear and concise, we restrict ourselves to the case of a continuous flow in time, but we believe that our local convergence analysis can be adapted using recent techniques from \cite{FengTang2025}, which shows in particular the (discrete-in-time) decay of the energy of the usual scheme of \cite{BaoDu2004} that we adopt in our numerical simulations.  We also emphasize that our analysis for proving Theorem \ref{th:exp_convergence} can be adapted to other consistent discretizations such as finite differences.

\section{Discrete functional analysis} \label{sec:discrete_function_analysis}

\subsection{Interpolants}
We introduce both the pointwise interpolant $P_h : \mathcal{C}^0\left(\overline{\mathcal{D}}\right) \rightarrow X(\mathcal{M})$ and the mean interpolant $\Pi_h : L^p(\mathcal{D}) \rightarrow X(\mathcal{M})$ for $p \geq 1$, respectively defined by
\[ (P_h \varphi)_K = \varphi(x_K) \quad \text{and} \quad (\Pi_h \varphi)_K = \frac{1}{|K|}\int_{K} \varphi(x) \dd x  \]
for any $K \in \mathcal{F}$. Similarly, we can define the same interpolants on the diamond mesh $\mathcal{M}^*$, namely $P_h^\diamond : \mathcal{C}^0(\overline{\mathcal{D}}) \rightarrow X(\mathcal{M}^*)$ and $\Pi_h^\diamond : L^p(\mathcal{D}) \rightarrow X(\mathcal{M}^*)$, by the formulas
\[ (P_h^\diamond \varphi)_\sigma = \varphi(x_\sigma) \quad \text{and} \quad (\Pi_h^\diamond \varphi)_\sigma = \frac{d}{|\sigma|d_\sigma}\int_{\mathcal{V}_\sigma} \varphi(x) \dd x  \]
for any $\sigma \in \mathcal{E}$. Note that $\Pi_h$ and $\Pi_h^\diamond$ are the $L^2$ orthogonal projections on $X(\mathcal{M})$ and $X(\mathcal{M}^*)$ respectively. Also, from Sobolev injections, we know that $H^s(\mathcal{D})$ and $W^{1, p}(\mathcal{D})$ continuously inject into $\mathcal{C}^0(\overline{\mathcal{D}})$ for $s>d/2$ and $p > d$, so that $P_h$ and $P_h^\diamond$ are well defined on these spaces. We compare the interpolants  $\Pi_h$ and $P_h$ in the following proposition, which is an improvement of \cite[Lemma 3.5]{Chauleur2026}.

\begin{proposition}\label{pr:diff_Pih_Ph}
  For all $p > d$, there exists $C_p > 0$ such that, for all admissible meshes~$\mathcal{M}$ and all $\varphi \in W^{1, p}(\mathcal{D})$, we have
  $$
  \|\Pi_h \varphi - P_h \varphi\|_{L^p} \leq C_p h \|\nabla \varphi\|_{L^p}.
  $$
\end{proposition}
\begin{proof}
 For any $K \in \mathcal{M}$, we write from Taylor expansion that
  \begin{align*}
    \left|\frac{1}{|K|}\int_{K}\varphi(y) \dd y - \varphi(x_K)\right| &= \left|\frac{1}{|K|}\int_{0}^{1} \int_K \nabla \varphi(ty + (1-t)x_K) \cdot (y - x_K) \dd y \dd t\right|,
    \\
    &\leq \frac{\diam(K)}{|K|}\int_{0}^{1} \int_K |\nabla \varphi(ty + (1-t)x_K)| \dd y \dd t.
  \end{align*}
  Defining $z = t y + (1-t)x_K$, we have $\dd z = t^d \dd y$, so from such change of variable,
  $$
  \int_K |\nabla \varphi(ty + (1-t)x_K)| \dd y = \frac{1}{t^d}\int_{K_t} |\nabla \varphi(z)| \dd z,
  $$
   for any $t \in (0,1]$, where $K_t \coloneqq t K + (1-t)x_K$. We also note that $|K_t| = t^d |K|$ so that, using Hölder's inequality, and as $K_t \subset K$ since $K$ is convex,
  \begin{align*}
    \left|(\Pi_h \varphi)_K - (P_h \varphi)_K\right| &\leq \frac{\diam(K)}{|K|} \int_{0}^{1} \frac{1}{t^d} \int_{K_t} |\nabla \varphi(z)| \dd z \dd t
    \\
    &\leq \frac{\diam(K)}{|K|}\int_0^1 \frac{|K_t|^{1-1/p}}{t^d} \|\nabla \varphi\|_{L^p(K_t)} \dd t
    \\ 
    &\leq \diam(K)|K|^{-1/p} \|\nabla \varphi\|_{L^p(K)} \int_{0}^{1} t^{-d/p} \dd t 
    \\
    &\leq C_p \diam(K) |K|^{-1/p} \|\nabla \varphi\|_{L^p(K)}
  \end{align*}
  since $p>d$. Taking this inequality to the power $p$, multiplying by $|K|$ and summing over $K \in \mathcal{F}$ gives the desired result.
\end{proof}

We also state and prove convergence properties of such interpolants.

\begin{proposition}\label{pr:approx_interpolants}
  For all $p \geq 1$, there exists $C_p > 0$ such that, for all admissible meshes $\mathcal{M}$ and $\varphi \in W^{1, p}(\mathcal{D})$, we have 
  $$
  \|\Pi_h \varphi - \varphi\|_{L^p} \leq C_p h \|\nabla \varphi\|_{L^p}.
  $$
  Moreover, if $p > d$, there also exists $C_p > 0$ such that, for all admissible meshes $\mathcal{M}$
  $$
  \|P_h \varphi - \varphi \|_{L^p} \leq C_p h \|\nabla \varphi\|_{L^p}.
  $$
\end{proposition}
\begin{proof}
  For any $K \in \mathcal{F}$, the Poincaré--Steklov inequality (see for example \cite[Lemma 3.24]{ErnGuermond2021}) provides 
  $$
  \|\Pi_h \varphi - \varphi\|_{L^p(K)}^p \leq C_p \diam(K)^p \|\nabla \varphi\|_{L^p(K)}^p,
  $$
  and summing over $K\in \mathcal{F}$ gives the first estimate. The second estimate immediately follows from the first one along Proposition \ref{pr:diff_Pih_Ph}.
\end{proof}

In the spirit of the proof of Proposition \ref{pr:diff_Pih_Ph}, the following lemma provides a sharp estimate of Taylor reminders on diamond cells, which will prove to be useful throughout the paper.

\begin{lemma}\label{lemma:diamond_integration_lemma}
  Let $K \in \mathcal{F}$ and $\sigma \in \mathcal{E}_K$, then the following estimate holds
  $$
  \left|\int_{\sigma}\int_{0}^{1}t^r \varphi(tx + (1-t)x_K) \dd t \dd x\right| \leq C_{p,r} \frac{|\sigma|^{1-1/p}}{d_{\sigma, K}^{1/p}} \|\varphi\|_{L^p(\mathcal{V}_{\sigma, K})},
  $$
  for any $\varphi$ in $L^p(\mathcal{V}_{\sigma, K})$, $r\in \R$ and $p\geq 1$ such that $p>\frac{d}{r+1}$.
\end{lemma}
\begin{proof}
  Up to rotations and translations, we may assume that $\sigma = \{0\} \times \tilde{\sigma}$ is included in $\{0\} \times \R^{d-1}$ and that $x_K = (-\alpha, 0_{d-1})$ with $\alpha>0$, as in the proof of \cite[Theorem 9.4]{EymardGallouetHerbin2000}. As such, we have $d_{\sigma, K} = \alpha$. We now define the change of variables $z(t, x) = ((t-1)\alpha, tx)$, which is a diffeomorphism from $(0, 1) \times \tilde{\sigma}$ to $\mathcal{V}_{\sigma, K}$. We then compute that, for $z=(z_1,\cdots,z_d)$ in $\mathcal{V}_{\sigma, K}$, $t = 1 + z_1/\alpha$ and 
\[    \dd z = t^{d-1}\alpha \dd t \dd x = \alpha\left(1 + \frac{z_1}{\alpha}\right)^{d-1} \dd t \dd x. \]
  We therefore write that 
  $$
  \int_{\sigma}\int_{0}^{1} t^r \varphi(tx + (1-t)x_K) \dd t \dd x = \frac{1}{\alpha}\int_{\mathcal{V}_{\sigma, K}} \left(1 + \frac{z_1}{\alpha}\right)^{r-d+1} \varphi(z) \dd z.
  $$
  Using Hölder's inequality with $\frac{1}{p}+\frac{1}{q}=1$, we obtain
  \begin{align*}
    \left|\int_{\sigma}\int_{0}^{1} t^r \varphi(tx + (1-t)x_K) \dd t \dd x\right| &\leq \frac{1}{\alpha}\left(\int_{\mathcal{V}_{\sigma, K}} \left(1 + \frac{z_1}{\alpha}\right)^{q(r-d+1)} \dd z \right)^{1/q} \|\varphi\|_{L^p(\mathcal{V}_{\sigma, K})}
    \\
    &\leq \frac{1}{\alpha }\left(\int_{\sigma}\int_{0}^{1} t^{q(r-d+1)}\alpha t^{d-1} \dd t \dd x \right)^{1/q} \|\varphi\|_{L^p(\mathcal{V}_{\sigma, K})},
  \end{align*}
  where we have used the change of variables in the opposite direction. We then compute 
  $$
  \int_{\sigma}\int_{0}^{1} t^{q(r-d+1)}\alpha t^{d-1} \dd t \dd x = \frac{|\sigma|\alpha}{d - q(d-1-r)},
  $$
  whenever $q(d-1-r) < d$. Recalling that $\frac{1}{p} + \frac{1}{q} = 1$, such condition is equivalent to $p > \frac{d}{r+1}$ from elementary algebraic manipulations.
\end{proof}

We now provide continuity estimates for the aforementioned interpolants.

\begin{proposition}\label{pr:12_boundedness_Ph}
 For all $\Xi \in (0, 1]$ and $p>d$, there exists $C_{p, \Xi} > 0$ such that, for all $\Xi$-regular admissible meshes and $\varphi \in W_0^{1,p}(\mathcal{D})$, we have
  $$
  |P_h \varphi|_{1, p} \leq C_{p, \Xi} \|\nabla \varphi\|_{L^p}.
  $$
\end{proposition}
\begin{proof}
  Let $K \in \mathcal{M}$, $\sigma \in \mathcal{E}_K$ and $x \in \sigma$. We start by assuming that $\sigma = K|L \in \mathcal{E}_{\mathrm{int}}$. From the Taylor expansion, we obtain
\begin{equation} \label{eq:Taylor_first_order}
 \varphi(x_K) - \varphi(x) = \int_{0}^{1} \nabla \varphi(tx_K + (1-t)x) \cdot (x_K - x) \dd t .
\end{equation} 
  Taking the average on $x \in \sigma$ and the difference with the same equation on the adjacent cell $L$, we obtain $\varphi(x_K) - \varphi(x_L) = (I_{\sigma, K} - I_{\sigma, L})/|\sigma|$, where 
  $$
  I_{\sigma, K} \coloneqq \int_\sigma \int_{0}^{1} \nabla \varphi(tx_K + (1-t)x) \cdot (x_K - x) \dd t \dd x.
  $$
  From Lemma \ref{lemma:diamond_integration_lemma}, we have
\[    |I_{\sigma, K}| \leq \diam(K) \int_\sigma \int_{0}^{1} |\nabla \varphi(tx_K + (1-t)x)| \dd t \dd x \leq C_p\frac{\diam(K)|\sigma|^{1-1/p}}{d_{ \sigma, K}^{1/p}}\|\nabla \varphi\|_{L^p(\mathcal{V}_{\sigma, K})},\]
therefore using \eqref{eq:mesh_regularity} we infer
\begin{align*}
  |\varphi(x_K)-\varphi(x_L)|^p &\leq \frac{C_p}{|\sigma|^p}\left(\frac{\diam(K)^p|\sigma|^{p-1}}{d_{\sigma, K}}\|\nabla \varphi\|_{L^p(\mathcal{V}_{\sigma, K})}^p + \frac{\diam(L)^p|\sigma|^{p-1}}{d_{\sigma, L}}\|\nabla \varphi\|_{L^p(\mathcal{V}_{\sigma, L})}^p\right),
  \\
  &\leq \frac{C_{p, \Xi}}{|\sigma|}h^{p-1}\|\nabla \varphi\|_{L^p(\mathcal{V}_{\sigma})}^p.
\end{align*}
Using once again the regularity assumption \eqref{eq:mesh_regularity} on $\mathcal{M}$, we obtain 
\begin{equation} \label{eq:estimate_derivative_Lp}
  \frac{|\sigma|}{d_{\sigma}^{p-1}}|D_\sigma P_h \varphi|^p \leq C_{p, \Xi} \|\nabla \varphi\|_{L^p(\mathcal{V}_\sigma)}^p.
\end{equation}
The case $\sigma \in \mathcal{E}_K \cap \mathcal{E}_{\mathrm{ext}}$ is treated similarly: Taylor expansion \eqref{eq:Taylor_first_order} holds with $\varphi(x)=0$ for $x\in\sigma\subset \partial \mathcal D$, so taking the average over $\sigma$, we get that $\varphi(x_K) = I_{\sigma, K}/|\sigma|$, hence an identical reasoning leads to the estimate \eqref{eq:estimate_derivative_Lp}. Taking the sum over all interior and exterior edges, we obtain the desired result as $\cup_{\sigma \in \mathcal{E}} \mathcal{V}_\sigma = \mathcal{D}$ up to a set of measure zero.
\end{proof}

We now prove a stability and convergence result for the interpolation operator on the diamond mesh $I_h$. 

\begin{proposition}\label{pr:l2_stability_I}
  For all $\Xi \in (0, 1]$ and $p \geq 1$, there exists $C_{p, \Xi} > 0$ such that, for all $\Xi$-regular admissible meshes and $u \in X(\mathcal{M})$, we have
\[  \|I_h u\|_{L^p} \leq C_{p, \Xi} \|u\|_{L^p} \quad \text{and} \quad \|I_h u - u\|_{L^p} \leq \frac{h}{d^{1/p}} |u|_{1, p}. \]
\end{proposition}
\begin{proof}
The first estimate follows from 
\[ \|I_h u\|_{L^p}^p = \sum_{\sigma \in \mathcal{E}} \frac{|\sigma|d_\sigma}{d} |I_\sigma u|^p \leq C_p \sum_{K \in \mathcal{F}} \sum_{\sigma \in \mathcal{E}_K} \frac{|\sigma|d_\sigma}{d} |u_K|^p \leq C_{p,\Xi} \sum_{K \in \mathcal{M}} \left(\sum_{\sigma \in \mathcal{E}_K} \frac{|\sigma|d_{\sigma, K}}{d}\right)|u_K|^p   \]
which gives the result since $\sum_{\sigma \in \mathcal{E}_K} \frac{|\sigma|d_{\sigma, K}}{d} = |K|$ as $K = \cup_{\sigma \in \mathcal{E}_K} \mathcal{V}_{\sigma, K}$ up to a set of measure zero. For the second estimate, one can remark that for all $K \in \mathcal{F}$ and $\sigma \in \mathcal{E}_K$, $I_h u - u$ is constant on $\mathcal{V}_{\sigma, K}$ and 
  $$
  (I_h u - u)_{|\mathcal{V}_{\sigma, K}} = \frac{1}{2}D_{\sigma, K} u,
  $$
  so that $|I_h u - u| \in X(\mathcal{M}^*)$ with $|I_h u - u|_{\sigma} = \frac{1}{2} |D_{\sigma} u|$.
  Therefore, we have 
\[    \|I_h u - u\|_{L^p}^p = \frac{1}{2^p} \sum_{\sigma \in \mathcal{E}} |D_\sigma u|^p \frac{\sigma d_\sigma}{d}= \frac{1}{d 2^p} \sum_{\sigma \in \mathcal{E}} d_\sigma^p \frac{|\sigma|}{d_\sigma^{p-1}} |D_\sigma u|^p \leq \frac{h^p}{d} |u|_{1, p}^p \]
  as $d_\sigma \leq 2 h$ for all $\sigma \in \mathcal{E}$, which ends the proof.
\end{proof}

Similar properties hold for the interpolants $P_h^\diamond$ and $\Pi_h^\diamond$ defined on the diamond mesh $\mathcal{M}^*$, and we only state those which will be useful later on, the proofs being nearly identical to the previous ones.

\begin{proposition} \label{pr:diff_Pih_Ph_diamond}
For all $p > d$, there exists $C_p>0$ such that, for all admissible meshes and $\varphi \in W^{1, p}(\mathcal{D})$, we have 
\[   \|\Pi_h^\diamond \varphi - P_h^\diamond \varphi\|_{L^p} \leq C_p h \|\nabla \varphi\|_{L^p} \quad \text{and} \quad \|P_h^\diamond \varphi - \varphi \|_{L^p} \leq C_p h \|\nabla \varphi\|_{L^p}.\]
\end{proposition}

We finally prove some relations between projections on the primal and diamond mesh. 
\begin{proposition}\label{pr:Ph_Ph_diamond}
  For all $p>d$, there exists $C_p>0$ such that, for all admissible meshes and $\varphi \in W^{1, p}_0(\mathcal{D})$, we have 
  $$
  \|I_h P_h \varphi - P_h^\diamond \varphi\|_{L^p} \leq C_{p} h \|\nabla \varphi\|_{L^p}.
  $$
\end{proposition}
\begin{proof}
  We write 
  $$
  \|I_h P_h \varphi - P_h^\diamond \varphi\| \leq \|I_h P_h \varphi - \Pi_h^\diamond \varphi\|_{L^p} + \|\Pi_h^\diamond \varphi - P_h^\diamond \varphi\|_{L^p},
  $$
  so that we only need to estimate $\|I_h P_h \varphi - \Pi_h^\diamond \varphi\|_{L^p}$ thanks to Proposition \ref{pr:diff_Pih_Ph_diamond}. Let $\sigma \in \mathcal{E}$. Proceeding as in the proof of Proposition \ref{pr:diff_Pih_Ph} by integrating over $\mathcal{V}_\sigma$ instead of $K$, we have
  $$
  |(\Pi_h^\diamond \varphi)_{\sigma} - \varphi(x_K)| \leq C_p \diam(\mathcal{V}_\sigma) |\mathcal{V}_\sigma|^{-1/p} \|\nabla \varphi\|_{L^p(\mathcal{V}_\sigma)},
  $$
  where $K$ is such that $\sigma \in \mathcal{E}_K$. If $\sigma = K|L$ is an interior edge, taking the average with the same inequality in $L$ gives
  $$
  |(\Pi_h^\diamond \varphi)_\sigma - I_\sigma P_h \varphi| \leq C_p \diam(\mathcal{V}_\sigma) |\mathcal{V}_\sigma|^{-1/p}\|\nabla \varphi\|_{L^p(\mathcal{V}_\sigma)}.
  $$
  On the other hand if $\sigma \in \mathcal{E}_K \cap \mathcal{E}_{\text{ext}}$, we infer
\[    |(\Pi_h^\diamond \varphi)_\sigma - I_\sigma P_h \varphi| = \left|(\Pi_h^\diamond \varphi)_\sigma - \frac{1}{2}\varphi(x_K)\right| \leq |(\Pi_h^\diamond \varphi)_{\sigma} - \varphi(x_K)| + \frac{1}{2}|\varphi(x_K)|. \]
  Then, as written above, we have first
  $$
  |(\Pi_h^\diamond \varphi)_\sigma - \varphi(x_K)| \leq C_p \diam(\mathcal{V}_\sigma)|\mathcal{V}_\sigma|^{-1/p}\|\nabla \varphi\|_{L^p(\mathcal{V}_\sigma)}.
  $$
  Moreover, using the same computations as in the proof of Proposition \ref{pr:12_boundedness_Ph} and the fact that $\sigma$ is an exterior edge, we have 
  \begin{align*}
    |\varphi(x_K)| &= \left|\frac{1}{|\sigma|}\int_\sigma \int_{0}^{1} \nabla \varphi(tx_K + (1-t)x) \cdot (x_K - x) dt dx\right|
    \\
    &\leq C_p \frac{\diam(\mathcal{V}_\sigma)}{d_{\sigma}^{1/p}|\sigma|^{1/p}}\|\nabla \varphi\|_{L^p(\mathcal{V}_\sigma)}
    \\
    &\leq C_p \diam(\mathcal{V}_\sigma) |\mathcal{V}_\sigma|^{-1/p} \|\nabla \varphi\|_{L^p(\mathcal{V}_\sigma)}
  \end{align*}
  Therefore, multiplying by $|\mathcal{V}_\sigma|$ and summing over $\sigma \in \mathcal{E}$ we obtain 
  \begin{align*}
    \|\Pi_h^\diamond \varphi - I_h P_h \varphi\|_{L^p}^p &\leq C_p \sum_{\sigma \in \mathcal{E}} \diam(\mathcal{V}_\sigma)^p \|\nabla \varphi\|_{L^p(\mathcal{V}_\sigma)}^p \leq C_{p} h^p \|\nabla \varphi\|_{L^p}^p.
  \end{align*}
\end{proof}

\section{Convergence of discrete minimizers} \label{sec:minimizers}

\subsection{Boundedness of the discrete energy} We turn to the analysis of the discrete energy \eqref{eq:discrete_GP_energy}. We start by proving continuity estimates for the discretized angular momentum and potential energy, then turn to the continuity and coercivity of the linear operator $H_h$.

\begin{proposition}\label{pr:Lz_self_adjoint}
 The operator $L_h : X(\mathcal{M}) \to X(\mathcal{M})$ is self adjoint for the $L^2$-product. Moreover, for all $\Xi \in (0, 1]$, there exists $C_\Xi > 0$ such that, for all $\Xi$-regular admissible meshes $\mathcal{M}$ and $u, v \in X(\mathcal{M})$,
\[  \left| \left( L_h u, v\right) \right| \leq C_{\Xi} \|u\|_{L^2} |v|_{1, 2} . \]
\end{proposition}
\begin{proof}
We start by remarking that for any $u \in X(\mathcal{M})$ and $K \in \mathcal{F}$, from the definition of $I_{\sigma} u$ along \eqref{eq:def_Lz}, we can rewrite
\[  (L_h u)_K = \frac{iu_K}{2|K|}\sum_{\sigma \in \mathcal{E}_K} |\sigma| R(x_\sigma) \cdot n_{\sigma, K} + \frac{i}{2|K|}\sum_{\sigma = K|L} |\sigma| u_L R(x_\sigma) \cdot n_{\sigma, K}, \]
  and the first term vanishes on the right-hand side as 
  \begin{align*}
    0 = \int_{K} \mathrm{div}(R)= \sum_{\sigma \in \mathcal{E}_K} \int_\sigma R(x) \cdot n_{\sigma, K} \dd x = \sum_{\sigma \in \mathcal{E}_K} |\sigma| R(x_\sigma) \cdot n_{\sigma, K},
  \end{align*}
  where we have used the linearity of $R$ to get that $\frac{1}{|\sigma|}\int_{\sigma} R(x) dx = R(x_\sigma)$. From \eqref{eq:def_Lz}, we thus have, using that $n_{\sigma, K} = -n_{\sigma, L}$ for $\sigma = K|L$, 
  \begin{align*}
    \left( L_h u, v \right) = -\frac{i}{2}\sum_{L \in \mathcal{F}} \sum_{\sigma = K|L} |\sigma| u_L \overline{v_K} R(x_\sigma) \cdot n_{\sigma, L}= \left(  u, L_h v \right),
  \end{align*}
   which proves the self-adjointness. To get the continuity estimate, we remark from \eqref{eq:def_Lz} that
  \begin{align*}
    \left( L_h u, v\right) &= i\sum_{K \in \mathcal{F}} \sum_{\sigma \in \mathcal{E}_K} |\sigma| (I_\sigma u) \overline{v_K} R(x_\sigma) \cdot n_{\sigma, K},
  \end{align*}
 for any $u,v \in X(\mathcal{M})$. For $\sigma = K|L$, we thus have the contribution
  $$
  |\sigma| I_\sigma u \left(n_{\sigma, K} \overline{v_K} + n_{\sigma, L} \overline{v_L}\right) \cdot R(x_\sigma) = |\sigma| (I_\sigma u) (D_\sigma \overline{v}) R(x_\sigma) \cdot n_\sigma,
  $$
  so that rearranging the sum, 
  $$
  \left( L_h u, v\right) = i\sum_{\sigma \in \mathcal{E}} |\sigma| (I_\sigma u)(D_\sigma \overline{v}) R(x_\sigma) \cdot n_\sigma.
  $$
  Therefore, by Cauchy-Schwarz inequality and Proposition~\ref{pr:l2_stability_I} we obtain 
  \begin{align*}
    |\left( L_h u, v\right)| &= \left|\sum_{\sigma \in \mathcal{E}} \left(\frac{|\sigma|^{1/2}}{d_\sigma^{1/2}}D_\sigma \overline{v}\right) \left(|\sigma|^{1/2}d_\sigma^{1/2} I_\sigma u\right) (R(x_\sigma) \cdot n_\sigma)\right|
    \\
    &\leq C \left(\sum_{\sigma \in \mathcal{E}} \frac{|\sigma|}{d_\sigma} |D_\sigma v|^2 \right)^{1/2} \left(\sum_{\sigma \in \mathcal{E}} \frac{|\sigma|d_\sigma}{d} |I_\sigma u|^2\right)^{1/2}
    \\
    &\leq C_{\Xi} \|u\|_{L^2} |v|_{1, 2}.
  \end{align*}  
\end{proof}

\begin{proposition}\label{pr:boundedness_pot}
For all $\Xi \in (0, 1]$, there exists $C_\Xi > 0$ such that, for all $\Xi$-regular admissible meshes $\mathcal{M}$ and for all $u,v \in X(\mathcal{M})$, 
  $$
  \left| \left( V_h u, v\right) \right|\leq C_{\Xi} \| V \|_{L^\infty} \|u\|_{L^2} \|v\|_{L^2}. 
  $$
\end{proposition}
\begin{proof}
We simply write from Cauchy-Scwharz inequality that
\[ \left|\left( V_h u, v\right) \right| = \left|\sum_{\sigma \in \mathcal{E}} V(x_\sigma) | \mathcal{V}_\sigma| I_\sigma u \overline{ I_\sigma v} \right| \leq \| V \|_{L^{\infty}} \| I_h u \|_{L^2} \| I_h v \|_{L^2}, \]
so the result follows from Proposition \ref{pr:l2_stability_I}.
\end{proof}

\begin{proposition}\label{pr:equivalence_h1}
For all $\Xi \in (0, 1]$, there exists $C_1, C_\Xi > 0$ such that, for all $\Xi$-regular admissible meshes and $u \in X(\mathcal{M})$, we have
\[  C_1 |u|_{1, 2}^2 \leq \left( H_h u, u \right) \leq C_{\Xi} |u|_{1, 2}^2. \]
\end{proposition}
\begin{proof}
The upper bound is a direct consequence of Proposition \ref{pr:Lz_self_adjoint}, Proposition \ref{pr:boundedness_pot} and discrete Poincaré inequality $\|u\|_{L^2} \leq C_{\mathcal{D}} |u|_{1, 2}$ (see e.g. \cite[Lemma 9.1]{EymardGallouetHerbin2000}), writing that
\[    \left( H_h u, u \right) = -\left( A_h u, u \right) + \left( V_h u, u\right) - \Omega\left( L_h u, u\right) \leq |u|_{1, 2}^2 + C_{\Xi} \|u\|_{L^2}^2 + C_{\Xi} \|u\|_{L^2}|u|_{1, 2} \leq C_{\Xi} |u|_{1, 2}^2. \]
 For the lower bound, we adapt the argument from \cite[Appendix A]{DodingHenning2024} in the discrete setting. We begin by computing 
\[    \sum_{\sigma \in \mathcal{E}} d_\sigma |\sigma| \left|\frac{D_\sigma u}{d_\sigma}n_\sigma + i \frac{\Omega}{2} R(x_\sigma) I_\sigma u\right|^2 
    = -\langle A_h u, u\rangle - \Omega \langle L_h u, u\rangle + \frac{d \Omega^2}{4} \sum_{\sigma \in \mathcal{E}} \frac{|\sigma| d_\sigma}{d} |R(x_\sigma)|^2 |I_\sigma u|^2,\]
  where we have used that $L_h$ is self-adjoint. We thus deduce that 
  \begin{equation}\label{eq:H_prod_symmetric}
    \left( H_h u, u\right) = \sum_{\sigma \in \mathcal{E}} d_\sigma |\sigma| \left|\frac{D_\sigma u}{d_\sigma}n_\sigma + i \frac{\Omega}{2} R(x_\sigma) I_\sigma u\right|^2 + \sum_{\sigma \in \mathcal{E}} \frac{|\sigma|d_\sigma}{d} \left(V(x_\sigma) - \frac{d\Omega^2}{4} |R(x_\sigma)|^2\right) |I_\sigma u|^2.
  \end{equation}
  We now focus on the first part of the right hand side of \eqref{eq:H_prod_symmetric}. We write 
  $$
  \left|\frac{D_\sigma u}{d_\sigma} n_\sigma + i\frac{\Omega}{2}(I_\sigma u)R(x_\sigma)\right|^2 \geq \frac{|D_\sigma u|^2}{d_\sigma^2} - \Omega \frac{|D_\sigma u|}{d_\sigma} |I_\sigma u| |R(x_\sigma)| + \frac{\Omega^2}{4}|I_\sigma u|^2 |R(x_\sigma)|^2.
  $$
  Using Young's inequality $ab \leq \frac{1}{\delta}a^2 + \frac{\delta}{4}b^2$ with $a = \frac{|D_\sigma u|}{d_\sigma}$ and $b = \Omega |I_\sigma u| |R(x_\sigma)|$, we get
  $$
  \left|\frac{D_\sigma u}{d_\sigma} t_\sigma + i\frac{\Omega}{2}(I_\sigma u)x_\sigma\right|^2 \geq \left(\frac{\delta - 1}{\delta}\right) \frac{|D_\sigma u|^2}{d_\sigma^2} + (1-\delta) \frac{\Omega^2}{4}|I_\sigma u|^2 |R(x_\sigma)|^2,
  $$
  so that, using \eqref{eq:H_prod_symmetric}, we infer
  $$
  \left( H_h u, u \right) \geq \left(\frac{\delta - 1}{\delta}\right) |u|_{1, 2}^2 + \sum_{\sigma \in \mathcal{E}} \frac{|\sigma|d_\sigma}{d}\left(V(x_\sigma) - \delta \frac{d \Omega^2}{4}|R(x_\sigma)|^2\right)|I_\sigma u|^2.
  $$
  We then set $\delta = 1 + \kappa$ where $\kappa$ is given by \eqref{eq:assumption_V}, so that 
\[    \left( H_h u, u \right)\geq \frac{\kappa}{1 + \kappa} |u|_{1, 2}^2 + \sum_{\sigma \in \mathcal{E}} \frac{|\sigma|d_\sigma}{d}\left(V(x_\sigma) - (1+\kappa) \frac{d \Omega^2}{4}|R(x_\sigma)|^2\right)|I_\sigma u|^2 \geq \frac{\kappa}{1+\kappa} |u|_{1, 2}^2. \]
\end{proof}

\subsection{Convergence of the discrete energy}

Before proving the convergence of the discrete energy towards the continuous one, we begin by showing the consistency of the discretized Hamiltonian. This is given by the three following propositions. 

\begin{proposition}\label{pr:convergence_laplacian}
  For all $\Xi \in (0, 1]$, there exists $C_\Xi > 0$ such that, for all $\Xi$-regular admissible meshes $\mathcal{M}$ and for all $\varphi \in H^2(\mathcal{D})$ and $v \in X(\mathcal{M})$, we have
\[  \left|\left( \Pi_h \Delta \varphi - A_h P_h \varphi, v \right)\right| \leq C_\Xi h |v|_{1, 2} \|D^2 \varphi\|_{L^2} , \]
  where $D^2 \varphi$ denotes the Hessian matrix of $\varphi$.
\end{proposition}
\begin{proof}
The result is classical, see for instance \cite[Theorem 9.4]{EymardGallouetHerbin2000}.
\end{proof}

\begin{proposition}\label{pr:convergence_Lz}
  For all $\Xi \in (0, 1]$ and $p > d$, there exists $C_{p, \Xi} > 0$ such that, for all $\Xi$-regular admissible meshes $\mathcal{M}$ and for all $\varphi \in W^{1, p}_0(\mathcal{D})$ and $v \in X(\mathcal{M})$, we have
\[  \left| \left( \Pi_h L \varphi - L_h P_h \varphi, v \right) \right| \leq C_{p, \Xi} h |v|_{1, q} \|\nabla \varphi\|_{L^p}, \]
where $1/p + 1/q = 1$.
\end{proposition}

\begin{proof}
Our main objective is to estimate the circulation error in the definition of $L_h$. More precisely, for $K \in \mathcal{F}$ and $\sigma \in \mathcal{E}_K$, we define the circulation error $E_{\sigma, K}$ as
\[  \left(\Pi_h L \varphi - L_h P_h \varphi \right)_K = \frac{i}{|K|}\sum_{\sigma \in \mathcal{E}_K} |\sigma| \Bigl( \underbrace{\frac{1}{|\sigma|}\int_{\sigma} \varphi(x) R(x) \cdot n_{\sigma, K} - I_\sigma (P_h \varphi) R(x_\sigma) \cdot n_{\sigma, K} }_{\eqqcolon E_{\sigma, K} } \Bigl). \]
We begin by estimating the error on interior edges. Let $K \in \mathcal{F}$ and $x \in \sigma = K|L$. From the usual Taylor expansion we have
$$
\varphi(x)(R(x) \cdot n_{\sigma, K}) - \varphi(x_K) (R(x) \cdot n_{\sigma, K}) = \int_0^1 \nabla \varphi(tx + (1-t)x_K) \cdot (x-x_K) (R(x) \cdot n_{\sigma, K}) \dd t.
$$
so taking the average on $x \in \sigma$, we obtain 
\begin{equation}\label{eq:formula_R_Lz}
\frac{1}{|\sigma|}\int_{\sigma} \varphi(x) (R(x) \cdot n_{\sigma, K}) - \varphi(x_K)(R(x_\sigma) \cdot n_{\sigma, K}) = \frac{1}{|\sigma|}J_{ \sigma,K},
\end{equation}
where we have used that $x_\sigma = \frac{1}{|\sigma|}\int_\sigma x \dd x$ is the barycenter of $\sigma$, and defining
$$
J_{ \sigma, K} \coloneqq \frac{1}{|\sigma|}\int_\sigma \int_0^1 \nabla \varphi(tx + (1-t)x_K) \cdot (x-x_K) (R(x) \cdot n_{\sigma, K}) \dd t \dd x.
$$
Taking the mean of \eqref{eq:formula_R_Lz} on $K$ and $L \in \mathcal{F}$, we infer
$$
E_{\sigma, K} = \frac{1}{2|\sigma|}(J_{\sigma, K} - J_{\sigma, L}),
$$
where we have used that $I_\sigma(P_h \varphi) = \frac{1}{2}(\varphi(x_K) + \varphi(x_L))$ and $n_{\sigma, L} = -n_{\sigma, K}$. In order to estimate $E_{\sigma, K}$, we thus begin by estimating $J_{\sigma, K}$. We have, using Lemma \ref{lemma:diamond_integration_lemma} and thanks to \eqref{eq:mesh_regularity} that
\begin{align*}
  |J_{\sigma, K}| &\leq C \diam(K) \int_\sigma \int_0^1 |\nabla \varphi(tx + (1-t)x_K)| \dd t \dd x,  
  \\ 
  &\leq C_{p} \frac{\diam(K)|\sigma|^{1-1/p}}{d_{\sigma, K}^{1/p}} \|\nabla \varphi\|_{L^p(\mathcal{V}_{\sigma, K})},
  \\
  &\leq C_{p, \Xi} h d_\sigma^{-1/p} |\sigma|^{1-1/p} \|\nabla \varphi\|_{L^p(\mathcal{V}_{\sigma, K})},
\end{align*}
so that 
\begin{equation} \label{eq:bound_E_K_sigma}
|E_{\sigma, K}| \leq C_{p, \Xi} h d_\sigma^{-1/p} |\sigma|^{-1/p} \|\nabla \varphi\|_{L^p(\mathcal{V}_\sigma)},
\end{equation}
where we recall that $\mathcal{V}_\sigma = \mathcal{V}_{\sigma, K} \cup \mathcal{V}_{\sigma, L}$. The same reasoning allows to obtain a similar estimate for edges on the boundary $\sigma \in \mathcal{E}_{\mathrm{ext}}$, using that $\varphi(x)=0$ for $x \in \partial \mathcal{D}$ and that $\mathcal{V}_\sigma = \mathcal{V}_{\sigma, K}$ in such case. We may now estimate $\Pi_h L \varphi - L_h P_h \varphi$. For $v \in X(\mathcal{M})$, we have
\begin{align*}
    \left( \Pi_h L \varphi - L_h P_h \varphi, v \right) &= i\sum_{K\in\mathcal{F}} \sum_{\sigma \in \mathcal{E}_K} |\sigma| E_{\sigma, K} \overline{v_K},
    \\
    &= i\sum_{\sigma=K|L} |\sigma|(E_{\sigma, K}\overline{v_K} + E_{\sigma, L}\overline{v_L}) + i\sum_{\sigma \in \mathcal{E}_{\mathrm{ext}}} |\sigma| E_{\sigma, K} \overline{v_K}.
  \end{align*}
  Since $E_{\sigma, K} = -E_{\sigma, L}$, we infer 
  $$
  \left|\left( L_h \Pi_h \varphi - \Pi_h L \varphi, v \right)\right| \leq \sum_{\sigma \in \mathcal{E}} E_\sigma |\sigma| |D_\sigma v|,
  $$
  with $E_\sigma = |E_{\sigma, K}| = |E_{\sigma, L}|$. Using Hölder's inequality with $\frac{1}{p}+\frac{1}{q}=1$, by \eqref{eq:bound_E_K_sigma} we finally get
  \begin{align*}
    \left|\left(  \Pi_h L \varphi - L_h P_h \varphi , v \right)\right| &\leq \sum_{\sigma \in \mathcal{E}} \left(E_\sigma |\sigma|^{1/p} d_{\sigma}^{1/p}\right) \left(|D_\sigma v| \frac{|\sigma|^{1/q}}{d_{\sigma}^{1-1/q}}\right) \leq \left(\sum_{\sigma \in \mathcal{E}} R_\sigma^p |\sigma| d_{\sigma}\right)^{1/p} |v|_{1, q}\\ 
    &\leq C_{p, \Xi} h \left(\sum_{\sigma \in \mathcal{E}}\|\nabla \varphi\|_{L^p(\mathcal{V}_\sigma)}^p\right)^{1/p} |v|_{1, q},
  \end{align*}
  and we conclude using that $\cup_{\sigma \in \mathcal{E}}\mathcal{V}_\sigma = \mathcal{D}$ up to a set of measure zero.
\end{proof} 

\begin{proposition}\label{pr:convergence_V}
  For all $\Xi \in (0, 1]$ and $p > d$, there exists $C_{p, \Xi} > 0$ such that, for all $\Xi$-regular admissible meshes $\mathcal{M}$ and for all $\varphi \in W^{1, p}_0(\mathcal{D})$ and $v \in X(\mathcal{M})$, we have
\[  \left|\left( V_h P_h \varphi - \Pi_h V \varphi, v \right) \right| \leq C_{p, \Xi} h \|\varphi\|_{W^{1,p}} |v|_{1, q}, \]
where $1/p + 1/q = 1$.
\end{proposition}

\begin{proof}
  We directly compute, using that $P_h^\diamond(V \varphi)= (P_h^\diamond V) (P_h^\diamond\varphi)$,
  \begin{align*}
    \left( V_h P_h \varphi - \Pi_h V \varphi, v \right) &= \left( V_h P_h \varphi, v \right) - \left( V \varphi, v \right), 
    \\
    &= \left( (P_h^\diamond V) I_h P_h \varphi, I_h v \right) - \left( V \varphi, I_h v \right) + \left( V\varphi, I_h v - v \right),
    \\
    &= \left( (P_h^\diamond V) I_h P_h \varphi - \Pi_h^\diamond(V \varphi), I_h v \right) + \left( V\varphi, I_h v - v \right),
    \\
    &= \left( P_h^\diamond(V \varphi) - \Pi_h^\diamond(V \varphi), I_hv \right) + \left( (P_h^\diamond V) (I_h P_h \varphi - P_h^\diamond \varphi), I_h v \right) 
    \\
    &+ \left( V\varphi, I_h v - v \right),
  \end{align*}
  so by Hölder inequality
  \begin{align*}
  |\left( V_h P_h \varphi - \Pi_h V \varphi, v \right)| &\leq \|P_h^\diamond (V \varphi) - \Pi_h^\diamond(V \varphi)\|_{L^p} \|I_h v\|_{L^q}  
  \\
  &+ \|V\|_{L^\infty} \|I_h P_h \varphi - P_h^\diamond \varphi\|_{L^p} \|I_h v\|_{L^q}+ \|V\|_{L^\infty} \|\varphi\|_{L^p} \|I_h v - v \|_{L^q}.
  \end{align*}
  Note that $V \in W^{1, \infty}(\mathcal{D})$ provides that $V \varphi \in W^{1, p}(\mathcal{D})$. The result then follows from Proposition \ref{pr:l2_stability_I}, Proposition \ref{pr:diff_Pih_Ph_diamond} and Proposition \ref{pr:Ph_Ph_diamond}.
\end{proof}

We now have all the ingredients for the proof of the convergence of the discrete energy. 

\begin{proposition}\label{pr:convergence_energy}
Let $(\mathcal{M}_h)_{h>0}$ be a family of $\Xi$-regular meshes with $\Xi \in (0, 1]$. Then for all $\varphi \in H^1_0(\mathcal{D}) \cap H^s(\mathcal{D})$ with $s > d/2$, we have $E_h(P_h \varphi) \rightarrow E(\varphi)$ as $h\rightarrow 0$. Furthermore, if $s \geq 2$, then
  $$
  |E_h(P_h \varphi) - E(\varphi)| \leq C_{\Xi, \varphi} h,
  $$
  where $C_{\Xi, \varphi}$ depends only on $\varphi$ through its $H^2(\mathcal{D})$-norm.
\end{proposition}
\begin{proof}
  We begin with the second case $\varphi \in H^1_0(\mathcal{D}) \cap H^2(\mathcal{D})$ before treating the more general one using a density argument. Throughout the proof, we will often use without further mention that, for $s>d/2$, $H^s(\mathcal{D})$ injects continuously in $W^{1, p}(\mathcal{D})$ for some $p>d$ and that, for $p_2 > p_1$, we have the norm inequalities
  $$
  \|\cdot\|_{L^{p_2}} \leq C \|\cdot\|_{L^{p_1}} \; \text{and } |\cdot|_{1, p_2} \leq C_\Xi |\cdot|_{1, p_1}.
  $$
  We refer the reader to \cite[Remark 2.2]{BessemoulinChainaisFilbet2015} for the second inequality. Let $\varphi \in H^1_0(\mathcal{D}) \cap H^2(\mathcal{D})$. We write
  \begin{align}
    E(\varphi) - E_h(P_h \varphi) &= \frac{1}{2}\left( H \varphi, \varphi\right) - \frac{1}{2}\left( H_h P_h \varphi, P_h \varphi \right) + \frac{1}{4}\left(\|\varphi\|_{L^4}^4 - \|P_h \varphi\|_{L^4}^4\right), \nonumber
    \\
    &= \frac{1}{2}\left( \Pi_h H \varphi - H_h P_h \varphi, P_h \varphi \right) + \frac{1}{2}\left( H \varphi, \varphi - P_h \varphi \right) +  \frac{1}{4}\left(\|\varphi\|_{L^4}^4 - \|P_h \varphi\|_{L^4}^4\right).\label{eq:decomposition_energy_error}
  \end{align}
  For the first term of \eqref{eq:decomposition_energy_error}, using Proposition~\ref{pr:convergence_laplacian} and Proposition~\ref{pr:12_boundedness_Ph}, we have 
  \begin{equation} \label{eq:first_term_energy}
    |(\Pi_h \Delta \varphi - A_h P_h \varphi, P_h \varphi)| \leq C_\Xi h |P_h \varphi|_{1, 2} \|D^2 \varphi\|_{L^2} \leq C_\Xi h \|\varphi\|_{H^2}^2. 
  \end{equation}
  Then, choosing $p > d$ appropriately with $1/p + 1/q = 1$, Proposition~\ref{pr:convergence_Lz} gives 
  \begin{equation} \label{eq:second_term_energy}
    |(\Pi_h L \varphi - L_h P_h \varphi, P_h \varphi)| \leq C_\Xi h |P_h \varphi|_{1, q} \|\nabla \varphi\|_{L^p}\leq C_\Xi h \|\varphi\|_{H^2}^2. 
  \end{equation}
  Similarly, we find using Proposition~\ref{pr:convergence_V}
  \begin{equation}\label{eq:third_term_energy}
    |(\Pi_h V \varphi - V_h P_h \varphi, P_h \varphi)| \leq C_\Xi h \|\varphi\|_{H^2}^2.
  \end{equation}
  Gathering \eqref{eq:first_term_energy}, \eqref{eq:second_term_energy} and \eqref{eq:third_term_energy} yields
  $$
  |\left( \Pi_h H \varphi - H_h P_h \varphi, P_h \varphi \right)| \leq C_\Xi h \|\varphi\|_{H^2}^2,
  $$
  for the first term of \eqref{eq:decomposition_energy_error}. For the second term of \eqref{eq:decomposition_energy_error}, we simply write using Cauchy--Schwarz inequality and Proposition~\ref{pr:approx_interpolants} that
  \begin{align*}
    |(H \varphi, \varphi - P_h \varphi)| &\leq \|H \varphi\|_{L^2} \|\varphi - P_h \varphi\|_{L^2}\leq C h \|\varphi\|_{H^2}^2.
  \end{align*}
  Finally, for the last term of \eqref{eq:decomposition_energy_error}, we write $\|P_h \varphi\|_{L^4} \leq C \|P_h \varphi\|_{L^\infty} \leq C \|\varphi\|_{L^\infty} \leq C \|\varphi\|_{H^2}$. Together with Proposition~\ref{pr:approx_interpolants} and the continuous injection of $H^2(\mathcal{D})$ in $W^{1, 4}(\mathcal{D})$, this yields 
  \begin{align*}
    \left|\|\varphi\|_{L^4}^4 - \|P_h \varphi\|_{L^4}^4\right| &= \left|\|\varphi\|_{L^4} - \|P_h \varphi\|_{L^4}\right|(\|\varphi\|_{L^4} + \|P_h \varphi\|_{L^4})(\|\varphi\|_{L^4}^2 + \|P_h \varphi\|_{L^4}^2)
    \\
    &\leq C \|\varphi - P_h \varphi\|_{L^4} \|\varphi\|_{H^2}^3 
    \\
    &\leq C h \|\varphi\|_{H^2}^4,
  \end{align*}
  which allows us to conclude that 
  $$
  |E(\varphi) - E_h(P_h \varphi)| \leq C_{\Xi, \varphi} h.
  $$
  
  Now, let $\varphi \in H^1_0(\mathcal{D}) \cap H^s(\mathcal{D})$ with $s \in (d/2, 2)$. We proceed by a diagonal argument using the density of $H^2(\mathcal{D})$ into $H^s(\mathcal{D})$. Let $(\psi_n)_{n \in \N}$ be a sequence of $H^2(\mathcal{D})$ which converges to $\varphi$ in $H^s(\mathcal{D})$. We write, for all $h>0$ and $n \in \N$ that
  \begin{equation} \label{eq:energy_compactness}
    |E_h(P_h \varphi) - E(\varphi)| \leq |E_h(P_h \varphi) - E_h(P_h \psi_n)| + |E_h(P_h \psi_n) - E(\psi_n)|  + |E(\psi_n) - E(\varphi)|.
  \end{equation}  
  We begin by treating the first term in the right-hand side of \eqref{eq:energy_compactness}. We prove that for all $\varepsilon > 0$ there exists $N_1\in \N$ such that for all $n \geq N_1$ and for all $h>0$, we have
  \begin{equation}\label{eq:uniform_continuity}
  |E_h(P_h \varphi) - E_h(P_h \psi_n)| < \frac{\varepsilon}{3}.
  \end{equation}
  Indeed, we have from Proposition \ref{pr:equivalence_h1} and Proposition \ref{pr:12_boundedness_Ph} that 
  \begin{align*}
    |\left( H_h P_h \varphi, P_h \varphi \right) - \left( H_h P_h \psi_n, P_h \psi_n \right)| &= |\left( H_h P_h(\varphi - \psi_n), P_h(\varphi + \psi_n) \right)|,
    \\
    &\leq C_\Xi |P_h(\varphi - \psi_n)|_{1, 2} |P_h(\varphi + \psi_n)|_{1, 2},
    \\
    &\leq C_\Xi \|\varphi - \psi_n\|_{H^s} \|\varphi + \psi_n\|_{H^s}.
  \end{align*}
  Similarly, 
  \begin{align*}
    \left|\|P_h \varphi\|_{L^4}^4 - \|P_h \psi_n\|_{L^4}^4\right| &\leq \|P_h(\varphi - \psi_n)\|_{L^4}\left(\|P_h \varphi\|_{L^4} + \|P_h \psi_n\|_{L^4}\right)\left(\|P_h \varphi\|_{L^4}^2 + \|P_h \psi_n\|_{L^4}^2\right),
    \\
    &\leq C\|\varphi - \psi_n\|_{H^s}\left(\|\varphi\|_{H^s} + \|\psi_n\|_{H^s}\right)\left(\|\varphi\|_{H^s}^2 + \|\psi_n\|_{H^s}^2\right),
  \end{align*}
  and as these estimates are independent of $h$, we infer \eqref{eq:uniform_continuity} as claimed. Now, using the continuity of $E$ in $H^1(\mathcal{D})$, we directly get that there exists $N_2 \in \N$ such that for all $n \geq N_2$, $|E(\psi_n) - E(\varphi)| < \varepsilon/3$. Finally, for all $n \in \N$, $\psi_n \in H^2(\mathcal{D})$, so that there exists $h_0>0$ such that for all $h<h_0$, $|E_h(P_h \psi_n) - E(\psi_n)| < \varepsilon/3$. Combining these inequalities we conclude that, for all $\varepsilon > 0$, by choosing $n = \max(N_1, N_2)$, there exists $h_0 > 0$ such that for all $h<h_0$,
\[  |E_h(P_h \varphi) - E(\varphi)| < \varepsilon.\]
\end{proof}

\subsection{Convergence of discrete minimizers of the energy}

In this section, we consider a sequence of uniformly regular admissible meshes $\left(\mathcal{M}_n\right)_{n\in\N}$ such that $h_n \rightarrow 0$. We begin by proving a compactness and consistency property of the linear part $H_h$ of the energy.

\begin{proposition}\label{pr:compactness_consistency}
  Let $(u_n)_{n \in \N}$ be a sequence such that, for all $n \in \N$, $u_n \in X(\mathcal{M}_n)$. If there exists $C > 0$ such that
  $$
  \left( H_{h_n} u_n, u_n \right) \leq C,
  $$
  for all $n \in \N$, then, up to a subsequence, $(u_n)_{n \in \N}$ converges strongly in $L^2(\mathcal{D})$ to a limit $\psi \in H^1_0(\mathcal{D})$. Furthermore, for this subsequence and for all $\varphi \in H^1_0(\mathcal{D}) \cap H^2(\mathcal{D})$, we have 
  $$
  \lim_{n \rightarrow +\infty} \left( H_{h_n} u_n, P_{h_n} \varphi \right) = \left( H \psi, \varphi \right).
  $$
\end{proposition}
\begin{proof}
  For the first claim, we know from Proposition \ref{pr:equivalence_h1} that $|u_n|_{1, 2}$ is uniformly bounded. From classical compactness results (see for instance \cite[Theorem 14.2]{EymardGallouetHerbin2000}), we know that $(u_n)_n$ converges in $L^2(\mathcal{D})$ to some $\psi \in H^1_0(\mathcal{D})$ up to a subsequence. For the second claim, we use the fact that $H_{h_n}$ is self-adjoint so that 
  $$
  \left( H_{h_n} u_n, P_{h_n} \varphi \right) = \left( u_n, H_{h_n} P_{h_n} \varphi \right).
  $$
  for all $\varphi \in H^1_0(\mathcal{D}) \cap H^2(\mathcal{D})$. Thus, we have 
  $$
  \left( H_{h_n} P_{h_n} \varphi, u_n \right) - \left( H \varphi, \psi \right) = \left( H_{h_n} P_{h_n} \varphi - \Pi_{h_n} H \varphi, u_n \right) + \left( H \varphi, u_n - \psi \right).
  $$
  For the first term, we proceed as in the proof of Proposition~\ref{pr:convergence_energy} to obtain
  $$
  \left( H_{h_n} P_{h_n} \varphi - \Pi_{h_n} H \varphi, u_n \right) \leq C_{\Xi} h_n |u_n|_{1, 2} \|\varphi\|_{H^2}.
  $$
  The second term is then treated using $\|H\varphi\|_{L^2} \leq C\|\varphi\|_{H^2}$ and the convergence of $u_n$ to $\psi$ in $L^2(\mathcal{D})$.
\end{proof}

We now prove the convergence of discrete critical points towards continuous critical points.

\begin{proposition}\label{pr:cv_critical_points}
  Let $(\uu_n)_{n\in\N}$ be a sequence with $\uu_n \in \mathrm{Crit}_{h_n}$ so that, for all $n\in\N$, there exists $\mu_n \in \R$ such that $\uu_n$ satisfies 
\begin{equation} \label{eq:discrete_minimizer}
 H_{h_n} \uu_n +|\uu_n|^2 \uu_n = \mu_n \uu_n.
\end{equation}
  Furthermore, assume that $(E_{h_n}(\uu_n))_{n \in \N}$ is bounded above. Then, up to a subsequence, $\uu_n$ converges strongly in $L^2(\mathcal{D})$ to $\phi^* \in H^1_0(\mathcal{D})$ which satisfies 
\begin{equation} \label{eq:critical_limit}
  H \phi^* + |\phi^*|^2 \phi^* = \mu \phi^*,
\end{equation}
 for some $\mu \in \R$. Furthermore, we have $\lim_{n \rightarrow +\infty} E_{h_n}(\uu_n) = E(\phi^*)$.
\end{proposition}
\begin{proof}
  Since $(E_{h_n}(\uu_n))_{n \in \N}$ is bounded above, relation \eqref{eq:discrete_GP_energy} implies that $(H_{h_n}u_n^*,u_n^*)$ is also bounded above, and hence we know from Proposition \ref{pr:compactness_consistency} that $(\uu_n)_{n \in \N}$ converges in $L^2(\mathcal{D})$ up to a subsequence to $\phi^* \in H^1_0(\mathcal{D})$. Furthermore, for this subsequence, we have for all $\varphi \in H^1_0(\mathcal{D}) \cap H^2(\mathcal{D})$
  $$
  \lim_{n \rightarrow +\infty} \left( H_{h_n} \uu_n, P_{h_n} \varphi \right) = \left( H \phi^*, \varphi \right).
  $$  
  For the nonlinear term, we use discrete Sobolev inequalities from \cite{BessemoulinChainaisFilbet2015} to prove that $\uu_n$ converges strongly in $L^4$. Indeed, we have 
  $$
  \|\uu_n\|_{L^6} \leq C_\Xi |\uu_n|_{1, 2},
  $$
  so that, using Lebesgue interpolation, we obtain
  $$
  \|\uu_n - \phi^*\|_{L^4} \leq \|\uu_n - \phi^*\|_{L^2}^{1/4} \|\uu_n - \phi^*\|_{L^6}^{3/4} \leq C \|\uu_n - \phi^*\|_{L^2}^{1/4}\left(|\uu_n|_{1, 2} + \|\phi^*\|_{H^1}\right)^{3/4}.
  $$
  This implies in particular that $\uu_n$ converges in $L^3$ so that $|\uu_n|^2\uu_n$ converges to $|\phi^*|^2 \phi^*$ in $L^1(\mathcal{D})$. Therefore, 
  $$
  \lim_{n \rightarrow +\infty} \left( |\uu_n|^2 \uu_n, \varphi \right) = \left( |\phi^*|^2 \phi^*, \varphi \right),
  $$
  as $\varphi \in H^2(\mathcal{D}) \hookrightarrow L^\infty(\mathcal{D})$. Finally, from \eqref{eq:alternative_formula_mu} we have $\mu_n = 2E_{h_n}(\uu_n) + \frac{1}{2}\|\uu_n\|_{L^4}^4$ so that the sequence $(\mu_n)_{n \in \N}$ is bounded. Therefore, extracting another subsequence if necessary, there exists $\mu \geq 0$ such that $\mu_n \rightarrow \mu$. We deduce that, for all $\varphi \in H^1_0(\mathcal{D}) \cap H^2(\mathcal{D})$, we have
  $$
  \lim_{n \rightarrow +\infty}\mu_n \left( \uu_n, \varphi \right) = \mu \left( \phi^*, \varphi \right).
  $$
  Combining the three limits above together with the equality  \eqref{eq:discrete_minimizer}, we infer
\[  \left( H \phi^*, \varphi \right) +  \left( |\phi^*|^2 \phi^*, \varphi \right) = \mu \left( \phi^*, \varphi \right), \]
  for all $\varphi \in H^1_0(\mathcal{D}) \cap H^2(\mathcal{D})$, so $\phi^*$ solves \eqref{eq:critical_limit}. We now prove that $E_{h_n}(\uu_n)$ converges to $E(\phi^*)$. Taking the inner product of \eqref{eq:critical_limit} with $\phi^*$ and using $\|\phi^*\|_{L^2}=1$, we obtain
\[  \mu = 2E(\phi^*) + \frac{1}{2}\|\phi^*\|_{L^4}^4. \]
  Thus, from $\mu_n = 2E_{h_n}(\uu_n) + \frac{1}{2}\|\uu_n\|_{L^4}^4$ and $\lim_{n \rightarrow \infty} \mu_n = \mu$, we obtain 
  $$
  \lim_{n \rightarrow \infty} 2E_{h_n}(\uu_n) + \frac{1}{2}\|\uu_n\|_{L^4}^4 = 2E(\phi^*) + \frac{1}{2}\|\phi^*\|_{L^4}^4.
  $$
  We then conclude using the convergence of $\uu_n$ to $\phi^*$ in $L^4(\mathcal{D})$ shown earlier in the proof. Note that, from uniqueness of the limit $\mu = 2E(\phi^*) + \frac{1}{2}\|\phi^*\|_{L^4}^4$, we actually find that $(\mu_n)_{n \in \N}$ converges to $\mu$, depending on $\phi^*$, without having to extract another subsequence.
  \end{proof}

We now have all the elements to prove Theorem \ref{th:cv_ground_states}. We first define, for any $\varphi \in H^1_0(\mathcal{D}) \cap H^s(\mathcal{D})$ normalized in $L^2(\mathcal{D})$, the normalized interpolant
$$
\widetilde{P}_{h_n} \varphi \coloneqq \frac{P_{h_n}\varphi}{\|P_{h_n}\varphi\|_{L^2}}.
$$
Choosing $p > d$ such that $H^s(\mathcal{D})$ injects in $W^{1, p}(\mathcal{D})$ and using 
$$
\|P_{h_n}\varphi-\varphi\|_{L^2} \leq C\|P_{h_n}\varphi - \varphi\|_{L^p} \leq C h_n \|\nabla \varphi\|_{L^p},
$$ 
from Proposition~\ref{pr:approx_interpolants} as well as $E_{h_n}(P_{h_n} \varphi) \rightarrow E(\varphi)$ from Proposition~\ref{pr:convergence_energy}, we readily get that $E_{h_n}(\widetilde{P}_{h_n} \varphi) \rightarrow E(\varphi)$.

We are now able to prove the main result of the section.

\begin{proof}[Proof of Theorem \ref{th:cv_ground_states}]
  Let $(\uu_n)_{n \in \N}$ be a sequence of minimizers of $E_{h_n}$. From Proposition~\ref{pr:equivalence_h1}, we know that $E_{h_n}(\uu_n) \geq 0$ for all $n$. Furthermore, we have for any $\varphi \in H^1_0(\mathcal{D}) \cap H^s(\mathcal{D})$ normalized in $L^2(\mathcal{D})$ that
\begin{equation} \label{eq:normalized_interpolant_energy}
  E_{h_n}(\uu_n) \leq E_{h_n}\left(\widetilde{P}_{h_n} \varphi\right),
\end{equation}
  so that, from the convergence of the right-hand side of this bound, we get that $(E_{h_n}(\uu_n))_{n \in \N}$ is bounded. Furthermore, since for all $n\in\N$, $\uu_n$ is a minimizer of $E_{h_n}$ on $\mathbb{S}_{h_n}$, it satisfies \eqref{eq:discrete_minimizer}. Hence, from Proposition \ref{pr:cv_critical_points}, we conclude that the sequence $(\uu_n)_{n \in \N}$ converges in $L^2(\mathcal{D})$ up to a subsequence to some $\phi^* \in H^1_0(\mathcal{D})$, which solves \eqref{eq:critical_limit} and such that $\lim_{n \rightarrow \infty} E_{h_n}(\uu_n) = E(\phi^*)$. It now remains to prove that $\phi^*$ is a minimizer of $E$. Let $\varphi$ in $H^1_0(\mathcal{D}) \cap H^s(\mathcal{D})$ be normalized in $L^2(\mathcal{D})$, then  \eqref{eq:normalized_interpolant_energy} holds for all $n \in \N$, hence passing to the limit, we obtain $E(\phi^*) \leq E(\varphi)$. From the density of $H^1_0(\mathcal{D}) \cap H^s(\mathcal{D})$ into $H^1_0(\mathcal{D})$ and by continuity of $E$ in $H^1_0(\mathcal{D})$, we finally get that $E(\phi^*) \leq E(\varphi)$ for all $\varphi \in H^1_0(\mathcal{D})$ normalized in $L^2(\mathcal D)$, implying that $\phi^*$ is a minimizer of the energy on the $L^2$-unit sphere.
\end{proof}

\section{Normalized gradient flow and spectral gap} \label{sec:gradient_flow}
This section is divided in two parts. In Section~\ref{sec:convergence_gradient_flow}, we work on a fixed admissible mesh and study the long time behaviour of solutions to the gradient flow equation \eqref{eq:gradient_flow}. More specifically, we prove Theorem~\ref{th:exp_convergence}, giving the local exponential convergence of solutions to \eqref{eq:gradient_flow} when the initial data is close to a critical point satisfying Assumption~\ref{ass:quasi_coercivity}. Section~\ref{sec:spectral_gap} is then devoted to study the dependence of Assumption~\ref{sec:spectral_gap} on the mesh. In particular, we show that if $(u^*_n)_{n \in \N}$ is a sequence of discrete critical points converging to a continuous critical point $\phi^*$ as in Theorem~\ref{th:cv_ground_states} and if $\phi^*$ satisfies an equivalent of Assumption~\ref{ass:quasi_coercivity} in the continuous case, then $u^*_n$ satisfies Assumption~\ref{ass:quasi_coercivity} for $n$ large enough.

\subsection{Convergence}\label{sec:convergence_gradient_flow}
Throughout this section, the admissible mesh $\mathcal M$ is fixed and $\uu \in \mathrm{Crit}_h$ denotes a critical point of the discrete energy~$E_h$ such that $E_h'(\uu) = \mu_h \uu$ and which satisfies Assumption~\ref{ass:quasi_coercivity}. Our approach is inspired by the method introduced in \cite{FaouJezequel2018}, the main novelties being that $d\in \left\{ 2,3\right\}$ and $\Omega \neq 0$, so minimizers are generally not real-valued. Denoting by $u(t)$ a solution of the normalized gradient flow~\eqref{eq:gradient_flow}, a direct consequence is that, for $u(0)$ close to $\uu$, it may be expected that $u(t)$ will converge to $\uu$ only up to phase, requiring a careful treatment of the phase invariance of \eqref{eq:discrete_GP_energy}. 

We recall the notation $T_{\uu}^\C \mathbb{S}_h = \enstq{v \in X(\mathcal{M})}{ \langle v, \uu\rangle = \langle v, i\uu\rangle = 0}$, and  briefly describe our strategy of proof. For all $t \geq 0$, we want to find an angle $\theta(t)$ so that the projection $v(t) = u(t) - \langle u(t), e^{i\theta(t)}\uu \rangle e^{i\theta(t)} \uu$ converges exponentially to zero as $t \to \infty$. Such angle $\theta(t)$ is chosen so that $v(t)$ is not only orthogonal to $e^{i\theta(t)}\uu$, but also to $ie^{i\theta(t)}\uu$, hence $v(t) \in T_{\uu}^\C \mathbb{S}_h$ for all $t \geq 0$ so we can use the spectral gap property provided by Assumption~\ref{ass:quasi_coercivity} to infer exponential convergence. The first part of the proof is therefore devoted to rigorously establishing the existence of such a function $\theta$, while the second part analyzes the equation satisfied by the projected solution $v$.

We denote the unit circle by $\T = \R/(2\pi\Z)$ and the related projection map $\Pi_\T : \R \rightarrow \T$. We also introduce $\tilde{\mathcal{B}}_\delta(\uu) = \enstq{u \in X(\mathcal{M})}{ |u - \uu|_{1, 2} < \delta}$ for $\delta > 0$ and define the pseudo-distances (see also \cite{ZhangHenning2025})
\[ d_0 (u, v) \coloneqq \min_{\alpha \in \T}\|u - e^{i\alpha}v\|_{L^2} \quad \text{and} \quad d_1 (u, v) \coloneqq \min_{\alpha \in \T}|u - e^{i\alpha}v|_{1,2} \]
on $X(\mathcal{M})$ as well as the tubular neighborhood $\tilde{\mathcal{T}}_\delta(\uu) \coloneqq \{u \in \mathbb{S}_h \mid d_1(u, \uu) < \delta\}$. The corresponding neighborhoods on the unit sphere are given by $\mathcal{B}_\delta(\uu) \coloneqq \tilde{\mathcal{B}}_\delta(\uu) \cap \mathbb{S}_h$ and $\mathcal{B}_\delta(\uu) \coloneqq \tilde{\mathcal{B}}_\delta(\uu) \cap \mathbb{S}_h$.

Throughout this section, for a function $R : X(\mathcal{M}) \rightarrow X(\mathcal{M})$ and an integer $p\in\N$, we say that $R(v) = \mathcal{O}(|v|_{1, 2}^p)$ if there exists $\delta > 0$ and $C \geq 0$ such that, for all $|v|_{1, 2} < \delta$, we have $\|R(v)\|_{L^2} \leq C |v|_{1, 2}^p$. Similarly, for $R : X(\mathcal{M}) \rightarrow \R$, we say that $R(v) = \mathcal{O}(|v|_{1, 2}^p)$ if there exists $\delta > 0$ and $C > 0$ such that, for all $|v|_{1, 2} < \delta$, we have $|R(v)| \leq C |v|_{1, 2}^p$. Moreover, all upper constants, generically denoted by $C>0$ for simplicity, may only depend on $|\uu|_{1, 2}$ as well as other fixed parameters of the model \eqref{eq:GP_energy}, but are uniform with respect to $h>0$. As such, if we have a sequence of discrete minimizers on different meshes converging to a continuous minimizer as in Theorem \ref{th:cv_ground_states}, all such constants are uniformly bounded.
 
Finally, from Proposition~\ref{pr:equivalence_h1}, the expression of $E''_h(\uu)$ in \eqref{eq:energy_hessian} and Assumption~\ref{ass:quasi_coercivity}, one concludes that the norm induced by $E''_h(\uu)-\mu\Id$ is equivalent to the discrete $H^1_0$-norm on $T_{\uu}^\C \mathbb{S}_h$, that is, there exists $K_1$ and $K_2$ such that 
\begin{equation}
  \label{eq:equivalenceS}
K_1 |v|_{1, 2}^2 \leq \langle (E_h''(\uu) - \mu\Id) v, v\rangle \leq K_2 |v|_{1, 2}^2,
\end{equation}
for all $v \in T_{\uu}^\C \mathbb{S}_h$. This fact will be used extensively in this section.

We start by proving the invariance of the tangent space $T_{\uu}^\C \mathbb{S}_h$ by phase shifts.

\begin{lemma}\label{lemma:invariance_theta_subspace}
    For any $\alpha \in \R$, we have $T_{e^{i\alpha}\uu}^\C \mathbb{S}_h = T_{\uu}^\C \mathbb{S}_h$. Furthermore, we have 
    $$
    \inf_{v \in T_{\uu}^\C \mathbb{S}_h\backslash\{0\}} \frac{\left\langle \left( E_h''(e^{i\alpha}\uu) - \mu_h \Id \right) v, v\right\rangle}{\|v\|_{L^2}^2} = \inf_{v \in T_{\uu}^\C \mathbb{S}_h\backslash\{0\}} \frac{\left\langle \left( E_h''(\uu) - \mu_h \Id \right) v, v\right\rangle}{\|v\|_{L^2}^2}.
    $$
\end{lemma}
\begin{proof}
    For the first statement, one computes that $\langle v, e^{i\alpha} \uu\rangle = \cos \alpha \langle v, \uu \rangle + \sin \alpha \langle v, i \uu \rangle$ and $\langle v, i e^{i\alpha} \uu \rangle = -\sin \alpha \langle v, \uu \rangle + \cos \alpha \langle v, i \uu \rangle$. Therefore, $\langle v, e^{i\alpha}\uu \rangle = \langle v, i e^{i\alpha} \uu \rangle = 0$ if and only if $\langle v, \uu\rangle = \langle v, iu^* \rangle = 0$.

    For the second statement, one readily computes that $E_h''(e^{i\alpha}\uu) v = e^{i\alpha} E_h''(\uu) e^{-i\alpha}v$. Furthermore, we remark that $e^{-i\alpha} v \in T_{\uu}^\C \mathbb{S}_h$ if and only if $v \in T_{e^{i\alpha}\uu}^\C \mathbb{S}_h = T_{\uu}^\C \mathbb{S}_h$. We conclude by combining these two observations.
\end{proof}

We now define $P_{\uu}^\C$ as the $L^2$-orthogonal projection on $T_{\uu}^\C \mathbb{S}_h$. From Lemma \ref{lemma:invariance_theta_subspace}, we therefore find that for any $\alpha \in \R$,
\begin{equation} \label{eq:proj_u_star_per}
P_{\uu}^\C u = u - \langle u, e^{i\alpha} \uu\rangle e^{i\alpha}\uu - \langle u, i e^{i\alpha}\uu \rangle i e^{i\alpha}\uu .
\end{equation}

\begin{lemma}\label{lemma:existence_theta}
    There exists $\rho > 0$ and a $\mathcal{C}^1$ map $\theta : \tilde{\mathcal{T}}_\rho(\uu) \rightarrow \T$ such that 
\begin{equation} \label{eq:theta_u_star}
    \langle u, ie^{i\theta(u)}\uu\rangle = 0,  
\end{equation}
    for all $u \in \tilde{\mathcal{T}}_\rho(\uu)$. Moreover, there exists $C>0$ such that, for all $u \in \tilde{\mathcal{T}}_\rho(\uu)$, we have 
\begin{equation} \label{eq:boundedness_minimizer_rotated}
 |u - e^{i\theta(u)}\uu|_{1, 2} \leq C d_1(u, \uu) \quad \text{and} \quad \|u - e^{i\theta(u)}\uu\|_{L^2} \leq C d_0(u, \uu).
\end{equation}
\end{lemma}
\begin{proof}
    We briefly sketch the proof. We begin by using an implicit function argument to build a map $\theta_0 : \tilde{\mathcal{B}}_\rho(\uu) \rightarrow \R$ and then use $\theta_0$ to define maps $\theta_\alpha : \tilde{\mathcal{B}}_\rho(e^{i\alpha} \uu) \rightarrow \R$ for any $\alpha \in \T$. We then prove that, for $\rho$ sufficiently small, these maps are all compatible when their images are projected onto $\T$, which allows us to define a map $\theta$ on the tubular neighborhood $\tilde{\mathcal{T}}_\rho(\uu)$.

    Let $F : X(\mathcal{M}) \times \R \rightarrow \R$ be given by 
    $$
    F(u, \theta) = \langle u, i e^{i\theta}\uu \rangle.
    $$
    We see that $F$ verifies the implicit function theorem condition around the pair $(\uu, 0)$, as $F(\uu, 0)=0$ and $\partial_\theta F(u, \theta) = -\langle u, e^{i\theta}\uu\rangle$, so that $\partial_\theta F(\uu, 0) = -1$. We thus deduce the existence of an open interval $I\subset\R$ containing $0$, $\rho>0$ and a $\mathcal C^1$ function $\theta_0$ from $\tilde{\mathcal{B}}_\rho(\uu)$ to $I$ such that
      \begin{equation*}
        \forall \theta \in I,\forall u\in \tilde{\mathcal{B}}_\rho(\uu),\qquad
        \left(
          F(u,\theta)=0 \Longleftrightarrow \theta=\theta_0(u)
        \right).
      \end{equation*}
      For later in the proof, we note that one may choose a smaller $\rho$ while keeping the same $I$ where uniqueness of solutions to $F(u, \theta) = 0$ is ensured.
\noindent
For $\alpha \in \R$ and $u \in \tilde{\mathcal{B}}_\rho(e^{i\alpha}\uu)$, we have $e^{-i\alpha}u \in \tilde{\mathcal{B}}_\rho(\uu)$. Therefore, we can define $\theta_\alpha : \tilde{\mathcal{B}}_\rho(e^{i\alpha}\uu) \rightarrow I + \alpha$ by 
    $$
    \theta_\alpha(u) = \theta_0(e^{-i\alpha}u) + \alpha.
    $$
    It is immediate that $\theta_\alpha(e^{i\alpha}\uu) = \alpha$ and $F(u, i e^{i\theta_\alpha(u)}\uu) = 0$ for all $u \in \tilde{\mathcal{B}}_\rho(e^{i\alpha}\uu)$. It remains to show that, if $u \in \tilde{\mathcal{T}}_\rho(\uu)$, then
    $$
    \theta(u) := \Pi_\T(\theta_\alpha(u))
    $$
    is independent of the choice of $\alpha \in \R$ such that $u \in \tilde{\mathcal{B}}_\rho(e^{i\alpha}\uu)$. This is equivalent to showing that, if $u \in \tilde{\mathcal{B}}_\rho(e^{i\alpha_1}\uu) \cap \tilde{\mathcal{B}}_\rho(e^{i\alpha_2}\uu)$, then 
    $$
    \theta_{\alpha_1}(u) = \theta_{\alpha_2}(u) \pmod{2\pi}.
    $$
    To simplify the problem, we define $\alpha = \alpha_2 - \alpha_1$ and write $w = e^{-i\alpha_1}u$. This allows us to reduce to proving that, for all $w \in \tilde{\mathcal{B}}_\rho(\uu) \cap \tilde{\mathcal{B}}_\rho(e^{-i\alpha}\uu)$, we have 
    $$
    \theta_{0}(w) = \theta_{0}(e^{i\alpha} w) - \alpha \pmod{2\pi}.
    $$
    Also, since $\theta_{0}(e^{i(\alpha + 2k\pi)}w) = \theta_0(e^{i\alpha}w)$ for $k \in \Z$, it is sufficient to prove that, for $\alpha \in [-\pi, \pi)$, we have 
\begin{equation} \label{eq:theta_0}
\theta_0(w) = \theta_0(e^{i\alpha} w) - \alpha,
\end{equation}
    if $w \in \tilde{\mathcal{B}}_\rho(\uu) \cap \tilde{\mathcal{B}}_\rho(e^{-i\alpha}\uu)$. As $I$ is a neighborhood of 0, we may choose $r \in (0, \pi)$ such that $(-r, r) \subset I$ and set
    $$
    m \coloneqq \min_{\pi \geq |\alpha| \geq r/2}|u^* - e^{-i\alpha}u^*|_{1, 2} > 0.
    $$ 
    Note that $r$ and $m$ do not depend on $\rho$. Up to choosing $\rho \leq m/2$, we now show that $\tilde{\mathcal{B}}_\rho(\uu) \cap \tilde{\mathcal{B}}_\rho(e^{-i\alpha}\uu)$ is empty whenever $\pi \geq |\alpha| \geq r/2$. Indeed, we have for such $\alpha$ and $w \in \tilde{\mathcal{B}}_\rho(\uu)$
\[      |u - e^{-i\alpha}\uu|_{1, 2} \geq |u^* - e^{-i\alpha}\uu|_{1, 2} - |u - \uu|_{1, 2}\geq m - \rho \geq \rho, \]
    so that $w \notin \tilde{\mathcal{B}}_\rho(e^{-i\alpha}\uu)$. Suppose now that $w \in \tilde{\mathcal{B}}_\rho(\uu) \cap \tilde{\mathcal{B}}_\rho(e^{-i\alpha}\uu)$ for $|\alpha|\leq\pi$. From the argument written just above, we have $|\alpha| < r/2$. Moreover, we may once again choose $\rho$ small enough so that $\theta_0(\tilde{\mathcal{B}}_\rho(\uu)) \subset (-r/2, r/2)$. From this, we obtain 
    $$
    \theta_0(e^{i\alpha}w) - \alpha \in (-r, r) \subset I.
    $$
    A simple computation shows that $F(w, \theta_0(e^{i\alpha}w) - \alpha) = 0$. Since by definition, $\theta = \theta_0(w)$ is the only solution in $I$ to $F(w, \theta) = 0$, we deduce that $\theta_0(e^{i\alpha}w) - \alpha = \theta_0(w)$, which proves~\eqref{eq:theta_0}.
    
    We now turn to inequalities \eqref{eq:boundedness_minimizer_rotated}. Suppose $u \in \tilde{\mathcal{T}}_\rho(\uu)$ and define $w_\alpha = e^{-i\alpha}u$ for all $\alpha \in \T$ such that $u \in \tilde{\mathcal{B}}_\rho(e^{i\alpha}\uu)$. Then, we have $w_\alpha \in \tilde{\mathcal{B}}_\rho(\uu)$ and $|u - e^{i\theta(u)}\uu|_{1, 2} = |w_\alpha - e^{i\theta_0(w_\alpha)}\uu|_{1,2}$, as $\theta(u) = \Pi_\T(\theta_0(w_\alpha)-\alpha)$. We now write 
    \begin{align*}
      |w_\alpha - e^{i\theta_0(w_\alpha)}\uu|_{1, 2} &\leq |w_\alpha - \uu|_{1, 2} + |\uu - e^{-i\theta_0(w_\alpha)}\uu|_{1, 2}
      \\
      &\leq |w_\alpha - \uu|_{1, 2} + C|\theta_0(w)|.
    \end{align*}
    where $C=C(|u^*|_{1,2})$. Then, since $\theta_0$ is $C^1$, we may choose $\rho$ small enough so that $\theta_0$ is globally Lipschitz on $\tilde{\mathcal{B}}_\rho(\uu)$. We infer that 
    $$
    |\theta_0(w_\alpha)| = |\theta_0(w_\alpha) - \theta_0(\uu)| \leq C|w_\alpha - \uu|_{1, 2},
    $$
    so that, using $|w_\alpha - \uu|_{1, 2} = |u - e^{i\alpha}\uu|_{1, 2}$ and $|w_\alpha - e^{i\theta_0(w_\alpha)}\uu|_{1, 2} = |u - e^{i\theta(u)}\uu|_{1, 2}$, we obtain 
    $$
    |u - e^{i\theta(u)}\uu|_{1, 2} \leq C|u - e^{i\alpha}\uu|_{1, 2},
    $$
    where $C$ is independent of $\alpha$. We may then obtain the leftmost inequality in \eqref{eq:boundedness_minimizer_rotated} by taking the minimum of the right-hand side over $\alpha \in \T$. The other inequality in \eqref{eq:boundedness_minimizer_rotated} is proven identically. For the remainder of the section, we note in passing that one can reduce $\rho$ while keeping the same constant $C$ in \eqref{eq:boundedness_minimizer_rotated}.
\end{proof}

 For $\rho$ as in Lemma~\ref{lemma:existence_theta}, we also define the mapping $r : \mathcal{T}_\rho(\uu) \rightarrow \R$
\begin{equation} \label{eq:def_r}
r(u) = \langle u, e^{i\theta(u)}\uu \rangle.
\end{equation}
Here, we recall that $\mathcal{T}_\rho(\uu) = \tilde{\mathcal{T}}_\rho(\uu) \cap \mathbb{S}_h$ is the tubular neighborhood of $\uu$ restricted to the $L^2$-sphere.

\begin{lemma}\label{lemma:projection_forumla}
    For $u \in \mathcal{T}_\rho(\uu)$, define $v = u - r(u)e^{i\theta(u)}\uu$. Then, up to choosing smaller $\rho$ if necessary, we have $v = P_{\uu}^\C u$ and 
    $$
    r(u) = \sqrt{1-\|v\|_{L^2}^2}.
    $$
    Moreover, there exists $C_1, C_2>0$ such that, for all $u \in \mathcal{T}_\rho(\uu)$, we have 
    \[ C_1 d_1(u, \uu) \leq |v|_{1, 2} \leq C_2 d_1(u, \uu), \quad \text{and} \quad C_1 d_0(u, \uu) \leq \|v\|_{L^2} \leq C_2 d_0(u, \uu).\]
\end{lemma}

\begin{proof}
    From the definition of $\theta$ and $r$ alongside \eqref{eq:proj_u_star_per} and \eqref{eq:theta_u_star}, we have 
    \begin{align*}
        v = u - \langle u, e^{i\theta(u)} \uu \rangle e^{i\theta} \uu= u - \langle u, e^{i\theta(u)} \uu \rangle e^{i\theta} \uu - \langle u, ie^{i\theta(u)} \uu \rangle ie^{i\theta} \uu = P_{\uu}^\C u.
    \end{align*}
    Using that $u$ and $\uu$ are normalized in $L^2(\mathcal{D})$, as well as $\langle v, e^{i \theta(u)}\uu \rangle = 0$, we infer 
    $$
    1 = r(u)^2 + \|v\|_{L^2}^2.
    $$
    Moreover, we have $r(u) - 1 = \langle u - e^{i\theta(u)}\uu, e^{i\theta(u)}\uu\rangle$, so that $|r(u) - 1| \leq \|u - e^{i\theta(u)}\uu\|_{L^2}$ by Cauchy-Schwarz inequality. From Lemma~\ref{lemma:existence_theta}, we infer $\|u - e^{i\theta(u)}\uu\|_{L^2} \leq C\rho$, so that we may choose $\rho$ small enough to ensure $r(u) > 0$ for all $u \in \mathcal{T}_\rho(\uu)$. We deduce that
    $$
    r(u) = \sqrt{1 - \|v\|_{L^2}^2}.
    $$
    For the final inequalities, we first write 
    \begin{align*}
      |v|_{1, 2} &= |u - r(u) e^{i\theta(u)}\uu|_{1, 2}
      \\
      &\leq |u - e^{i\theta(u)}\uu|_{1, 2} + |r(u)-1||\uu|_{1, 2}
      \\
      &\leq C_2 d_1(u, \uu),
    \end{align*}
    where we have used that $|r(u)-1| \leq C|u - e^{i\theta(u)}\uu|_{1, 2}$ from above as well as Lemma~\ref{lemma:existence_theta}. For the other inequality, we have 
    \begin{align*}
      |v|_{1, 2} \geq |u - e^{i\theta(u)}\uu|_{1, 2} - |r(u)-1| |\uu|_{1, 2}.
    \end{align*}
    Then, using $r(u) = \sqrt{1-\|v\|_{L^2}^2}$, Poincaré's inéquality and $|v|_{1, 2} \leq C_1 d_1(u, \uu)$, we infer $|r(u)-1| \leq C_\rho |v|_{1, 2}^2 \leq C_\rho C_1 \rho |v|_{1,2}$ for all $u \in \mathcal{T}_\rho(\uu)$. Therefore, we obtain 
    $$
    |v|_{1, 2} \geq C_2 |u - e^{i\theta(u)}\uu|_{1, 2} \geq C_2 d_1(u, \uu).
    $$
    The inequalities for the $L^2$-norm are obtained in the same way.
  \end{proof}

The following asymptotic developments will be useful in the sequel. 

\begin{lemma}\label{lemma:big_o_hessian_and_mu}
    For $u,v \in X(\mathcal{M})$, it holds
    \begin{align}
        E_h(u + v) &= E_h(u) + \langle E_h'(u), v \rangle + \frac{1}{2}\langle E_h''(u) v, v \rangle + \mathcal{O}(|v|_{1, 2}^3),\label{eq:dl_E}
        \\
        E_h'(u + v) &= E_h'(u) + E_h''(u) v + \mathcal{O}(|v|_{1, 2}^2),\label{eq:dl_E_prime}
        \\
        \mu_h(u + v) &= \mu_h(u) + \mathcal{O}(|v|_{1, 2}).\label{eq:dl_mu}
    \end{align} 
\end{lemma}
\begin{proof}
    We begin with \eqref{eq:dl_E}. Since $H_h$ is self-adjoint, we write
    $$
    \frac{1}{2}\langle H_h (u+v), u + v \rangle = \frac{1}{2} \langle H_h u, u\rangle + \langle H_h u, v \rangle + \frac{1}{2}\langle H_h v, v \rangle.
    $$
    For the nonlinear part of the energy, we develop 
    \[ |u+v|^4 =|u|^4 + 4 \Re\left(|u|^2u\overline{v}\right) + 2\Re\left(|u|^2 v \overline{v}\right) + 4\Re\left(u \Re\left(\overline{u}v\right)\overline{v}\right) + R_{u, v} \]
    where $|R_{u, v}| \leq C(|u|+|v|)|v|^3$. Integrating over $\mathcal D$, we therefore obtain 
    $$
    \frac{1}{4}\|u+v\|_{L^4}^4 = \frac{1}{4}\|u\|_{L^4}^4 + \left\langle |u|^2 u, v\right\rangle + \frac{1}{2}\left\langle |u|^2 v + 2\Re\left(u \Re\left(\overline{u}v\right)\right), v \right\rangle + \mathcal{O}(|v|_{1, 2}^3),
    $$
    where we have used the discrete Sobolev inequalities $\|v\|_{L^p} \leq C |v|_{1, 2}$ for $p=3$ and $4$.

    We turn to \eqref{eq:dl_E_prime}. We recall that $E'(u+v) = H_h (u+v) + |u+v|^2 (u+v)$, hence the linear part is treated as before and we only focus on the nonlinear part, writing that
    \[ |u+v|^2(u+v) =|u|^2 u + \left(|u|^2 v + 2 u \Re(\overline{u}v)\right)  + R_{u, v}, \]
    where $|R_{u, v}| \leq C(|u| + |v|)|v|^2$. Using the discrete Sobolev inequalities $\|v\|_{L^p} \leq C |v|_{1, 2}$ for $p=4$ and $6$, we obtain $\|R_{u, v}\|_{L^2} = \mathcal{O}(|v|_{1, 2}^2)$, which gives \eqref{eq:dl_E_prime}.

    Finally, we prove \eqref{eq:dl_mu}. We have by definition of $\mu_h$ that 
    $$
    \mu_h(u+v) = \frac{\langle E_h'(u+v), u+v\rangle}{\|u+v\|_{L^2}^2}.
    $$
    For the numerator, we have $E_h'(u+v) = E_h'(u) + \mathcal{O}(|v|_{1, 2})$ from \eqref{eq:dl_E_prime}, so that $\langle E_h'(u+v), u + v\rangle = \langle E_h'(u), u\rangle + \mathcal{O}(|v|_{1, 2})$. For the denominator, we have $\|u+v\|_{L^2}^2 = \|u\|_{L^2}^2 + \mathcal{O}(|v|_{1, 2})$ from the Poincaré inequality $\|v\|_{L^2} \leq C|v|_{1,2}$.
\end{proof}

\begin{remark}\label{rem:constants_big_O}
  The proof above is given for expansions around a general vector $u \in X(\mathcal{M})$. However, we will only use these expansions around vectors of the form $e^{i\alpha}\uu$. As such, the constants involved in the remainders depend only on the critical point norm $|u^*|_{1, 2}$, the domain $\mathcal{D}$ and the mesh regularity $\Xi$.
\end{remark}

We note that, in \eqref{eq:dl_E_prime}, compared to the analysis performed in \cite{FaouJezequel2018}, only the $L^2$-norm of the remainder is controlled rather than the full discrete $H^1$-norm. This originates from the fact that, for $d=2$ or $3$, the space $H^1(\mathcal{D})$ is no longer an algebra as it does not inject continuously into $L^\infty(\mathcal{D})$. As a consequence, we cannot proceed with the same bootstrap argument used in \cite{FaouJezequel2018}. We instead prove the two following results which allow us to obtain {\it a priori} bounds on the discrete $H^1$-norm of the error thanks to the energy decreasing property of the gradient flow \eqref{eq:gradient_flow}.

\begin{proposition}\label{pr:coercivity_energy}
  Up to choosing a smaller $\rho$, we have
  \begin{equation}\label{eq:coercivity_energy}
    E_h(u) - E_h(\uu) \geq \frac{K_1}{4} |v|_{1, 2}^2,
  \end{equation}
  for all $u \in \mathcal{T}_\rho(\uu)$, where $v$ is given by Lemma \ref{lemma:projection_forumla} and $K_1$ by \eqref{eq:equivalenceS}.
\end{proposition}

\begin{proof}
  First, we recall that $u = r(u)e^{i\theta(u)}\uu + v$ and $r(u) = \sqrt{1-\|v\|_{L^2}^2}$ from Lemma~\ref{lemma:projection_forumla}. Then, using \eqref{eq:dl_E} from Lemma~\ref{lemma:big_o_hessian_and_mu}, we have 
  \begin{align*}
    E_h(u) &= E_h\left(e^{i\theta(u)}\uu + v + \left(\sqrt{1-\|v\|_{L^2}^2} -1\right) e^{i\theta(u)}\uu\right)
    \\
    &= E_h(e^{i\theta(u)}\uu) + \left\langle E_h'(e^{i\theta(u)}\uu), v \right\rangle - \frac{1}{2}\|v\|_{L^2}^2\left\langle E_h'(e^{i\theta(u)}\uu), e^{i\theta(u)}\uu\right\rangle 
    \\
    &+ \frac{1}{2}\left\langle E_h''(e^{i\theta(u)}\uu) v, v \right\rangle + \mathcal{O}(|v|_{1, 2}^3).
  \end{align*}
  Here, we note that the expansion is done around $e^{i\theta(u)}\uu$, so that the constants in the $\mathcal{O}(|v|_{1, 2}^3)$ remainder only depend on $|u^*|_{1, 2}$, as previously mentioned in Remark~\ref{rem:constants_big_O}. Then, using $E_h'(e^{i\theta(u)}\uu) = \mu e^{i\theta(u)} \uu$ and the fact that $v$ is orthogonal to $e^{i\theta(u)} \uu$, we obtain 
  $$
  E_h(u) = E_h(\uu) + \frac{1}{2}\left\langle \left(E_h''(e^{i\theta(u)}\uu) - \mu \Id\right) v, v \right\rangle + \mathcal{O}(|v|_{1, 2}^3).
  $$
  From \eqref{eq:equivalenceS}, we may now choose $\kappa > 0$ such that
  $$
  |\mathcal{O}(|v|_{1, 2}^3)| \leq \frac{1}{4}\left\langle \left(E_h''(e^{i\theta(u)}\uu) - \mu \Id\right) v, v \right\rangle,
  $$
  whenever $|v|_{1, 2} < \kappa$, implying 
  \begin{align*}
  E_h(u) - E_h(\uu) &\geq \frac{1}{4}\left\langle \left(E_h''(e^{i\theta(u)}\uu) - \mu \Id\right) v, v \right\rangle
  \\
  &\geq \frac{K_1}{4}|v|_{1, 2}^2,
  \end{align*}
  for all $|v|_{1, 2} < \kappa$. Finally, from Lemma~\ref{lemma:projection_forumla}, we may choose $\rho \leq \kappa/C_2$ so that $u \in \mathcal{T}_\rho(\uu)$ implies $|v|_{1, 2} < \kappa$.
\end{proof}

In the remainder of the section $t\mapsto u(t)$ is a solution to \eqref{eq:gradient_flow} with $u(0) \in \mathcal{T}_\rho(\uu)$ and $v(t) = P_{\uu}^{\C} u(t)$. We also define $T_* = \sup\{T \mid u(t) \in \mathcal{T}_\rho(\uu) \forall t \in [0, T]\}$, and introduce the notation 
$$
S_t := E_h''(e^{i\theta(u(t))}\uu) - \mu\Id,
$$
for all $t \in [0, T_*)$.
We note that, as was shown in the proof of Lemma~\ref{lemma:invariance_theta_subspace}, $S_t$ is always conjugate to $E_h''(\uu) - \mu\Id$.

\begin{proposition}\label{pr:bounded_A_norm}
  For all $\varepsilon > 0$, there exists $\delta > 0$ such that, if $u(0) \in \mathcal{T}_\delta(\uu)$, then $T_*=+\infty$ and for all $t \geq 0$,
  $$
  |v(t)|_{1,2} < \varepsilon.
  $$
\end{proposition}

\begin{proof}
 We briefly sketch the proof. We use Proposition~\ref{pr:coercivity_energy}, which allows us to use the fact that solutions to \eqref{eq:gradient_flow} have decreasing energy. We first use this property to show that, once starting close enough to $u^*$, the solution remains in $\mathcal{T}_\rho(\uu)$ for all time, where $\rho$ is given by Proposition \ref{pr:coercivity_energy}. Then, we use Proposition~\ref{pr:coercivity_energy} and the energy decreasing property to bound $|v(t)|_{1, 2}$ by $E_h(u(0))-E_h(\uu)$, which may be chosen arbitrarily small.
  
  From the fact that $E_h$ is $\mathcal{C}^1$ with respect to the $|\cdot|_{1, 2}$ norm, we know that there exists $C > 0$ such that, for all $\delta\in(0,\rho]$, if $u(0) \in \mathcal{T}_\delta(\uu)$, then 
  $$
  E_h(u(0)) - E_h(\uu) < C\delta.
  $$
  Then, using that the gradient flow \eqref{eq:gradient_flow} is energy decreasing, we infer that 
  \begin{equation}\label{eq:energy_decreasing}
    E_h(u(t)) - E_h(\uu) < C\delta
  \end{equation}
  for all $t\geq 0$. We now prove that, up to choosing a small enough $\delta$, we have that $u(t) \in \mathcal{T}_\rho(\uu)$ for all $t\geq 0$. We reason by contradiction. Assuming $u(0) \in \mathcal{T}_{\rho/2}(\uu)$ and that $T_* < +\infty$, we know by continuity that there must exist $t \in (0, T_*)$ such that $d_{1,2}(u(t), \uu) \in (\rho/2, \rho)$. Then, at time $t$, Proposition \ref{pr:coercivity_energy} holds, so that from Lemma~\ref{lemma:projection_forumla} we obtain 
\[    E_h(u(t)) - E_h(\uu) \geq \frac{K_1}{4} |v|_{1, 2}^2 \geq \frac{K C_1^2}{4} d_1(u(t), \uu)^2 \geq \frac{K_1 C_1^2}{16} \rho^2.\]
  However, we know that $E_h(u(t)) - E_h(\uu) < C\delta$, so that we may choose $\delta \leq \min(\rho/2, \frac{K_1 C_1^2}{16C}\rho^2)$ to reach a contradiction, implying $T_* = +\infty$ as claimed. Finally, using \eqref{eq:coercivity_energy} together with \eqref{eq:energy_decreasing}, we deduce that 
  \begin{align*}
    K|v(t)|_{1, 2}^2 \leq E_h(u(t)) - E_h(\uu) < C\delta,
  \end{align*}
  for all $t\geq 0$. We are then able to conclude choosing $\delta \leq \frac{K\varepsilon}{C}$.
\end{proof}

\begin{proposition}\label{pr:evolution_v}
  For all $t \in (0, T_*)$, it holds
\begin{equation} \label{eq:evolution_v}
    \partial_t v(t) = -P_{\uu}^\C S_t v(t) + \mathcal{O}(|v(t)|_{1, 2}^2),
\end{equation}
where the constants involved in the remainder only depend on $|\uu|_{1,2}$.
\end{proposition}

\begin{proof}
    From Proposition \ref{lemma:projection_forumla}, we have $v(t) = P_{\uu}^\C u(t)$ so that $\partial_t v(t) = P_{\uu}^\C \partial_t u(t)$. Furthermore, from \eqref{eq:gradient_flow} we infer
    \begin{equation}\label{eq:projected_d_t_u}
        \partial_t v(t) = - P_{\uu}^\C E_h'(u(t)) + \mu_h(u(t))v(t).
    \end{equation}
    We then write 
    \begin{align*}
        E_h'(u) &= E_h'\left(e^{i\theta(u)}\uu + v + (r(u)-1)e^{i\theta(u)}\uu\right)
        \\
        &= E_h' (e^{i\theta(u)}\uu) + E_h''(e^{i\theta(u)}\uu) v + \mathcal{O}(|v|_{1, 2}^2)
        \\
        &= \mu e^{i\theta(u)} \uu + E_h''(e^{i\theta(u)}\uu) v + \mathcal{O}(|v|_{1, 2}^2),
    \end{align*}
    where we have used Lemma~\ref{lemma:projection_forumla} and Lemma~\ref{lemma:big_o_hessian_and_mu}. Once again, we note that the expansion of $E_h'$ is done at $e^{i\theta(u)}\uu$, so that the constants involved in the remainder depend only on $|\uu|_{1, 2}$, as mentioned in Remark~\ref{rem:constants_big_O}. Therefore, we get that
    $$
    P_{\uu}^\C E_h'(u) = P_{\uu}^\C E_h''(e^{i\theta(u)}\uu) v + \mathcal{O}(|v|_{1, 2}^2).
    $$
    Finally, from Lemma \ref{lemma:big_o_hessian_and_mu} we obtain
    \[        \mu_h(u) = \mu_h(e^{i\theta(u)}\uu) + \mathcal{O}(|v|_{1, 2})= \mu + \mathcal{O}(|v|_{1, 2}), \]
    so that $\mu_h(u) v = \mu v + \mathcal{O}(|v|_{1, 2}^2)$, where, once again, the constants in the remainder only depend on $|\uu|_{1,2}$. We are then able to conclude using \eqref{eq:projected_d_t_u}.
\end{proof}


\begin{proof}[Proof of Theorem \ref{th:exp_convergence}]
    From Proposition \ref{pr:evolution_v}, taking the inner product of \eqref{eq:evolution_v} with $v(t)$, we obtain 
    $$
    \frac{1}{2}\frac{\dd}{\dd t}\|v(t)\|_{L^2}^2 = -\langle S_t v(t), v(t)\rangle + \mathcal{O}(|v(t)|_{1, 2}^3).
    $$
    Furthermore, from the equivalence of $\langle S_t v, v\rangle$ and $|v|_{1, 2}^2$, we infer 
    $$
    \frac{1}{2}\frac{\dd}{\dd t}\|v(t)\|_{L^2}^2 \leq -\left(1 - C|v(t)|_{1,2}\right)\langle S_t v(t), v(t)\rangle.
    $$
    Then, using Proposition \ref{pr:bounded_A_norm}, we know that for all $\varepsilon > 0$, there exists $\delta > 0$ such that if $u(0) \in \mathcal{T}_\delta(\uu)$, we have $C|v(t)|_{1,2} \leq \varepsilon/\lambda$ for all $t \geq 0$, where $\lambda>0$ is given by \eqref{eq:lambda}. We deduce 
    $$
    \frac{1}{2}\frac{\dd}{\dd t}\|v(t)\|_{L^2}^2 \leq -\left(1-\frac{\varepsilon}{\lambda}\right)\langle S_t v(t), v(t)\rangle.
    $$
    Therefore, using $\langle S_t v(t), v(t)\rangle \geq \lambda \|v(t)\|_{L^2}^2$ along Grönwall's inequality, we obtain 
    $$
    \|v(t)\|_{L^2} \leq e^{-(\lambda-\varepsilon)t}\|v(0)\|_{L^2}.
    $$
    Finally, from the $L^2$ inequalities given in Lemma~\ref{lemma:projection_forumla}, we obtain 
    $$
    d_0(u(t), \uu) \leq \frac{C_2}{C_1} e^{-(\lambda-\varepsilon)t}d_0(u(0), \uu).
    $$
\end{proof}

\subsection{Convergence of the spectral gap}\label{sec:spectral_gap}

We now analyze how the exponential convergence rate $\lambda$ given by \eqref{eq:lambda} behaves as the mesh diameter $h$ tends to zero. We place ourselves in the context of Theorem \ref{th:cv_ground_states}: $(\mathcal{M}_n)_{n \in \N}$ denotes a sequence of uniformly regular admissible meshes with $h_n \rightarrow 0$, while $(\uu_n)_{n \in \N}$ denotes a sequence of normalized minimizers in $\mathcal{U}_{h_n}$. As given by Theorem \ref{th:cv_ground_states}, we may choose this sequence so that $(\uu_n)_{n \in \N}$ converges in $L^2(\mathcal{D})$ to a continuous ground state $\uphi$. Since $\uphi$ and $\uu_n$ are respectively constrained minimizers of $E$ and $E_{h_n}$, there exists $\mu$ and $\mu_n$ such that $E'(\uphi) = \mu \uphi$ and $E_{h_n}'(\uu_n) = \mu_n \uu_n$. Moreover, we also recall that, for such a sequence, there exists $C > 0$ such that $|u_n^*|_{1, 2} \leq C$ and that $\mu_n$ converges to $\mu$ (see the proof of Proposition~\ref{pr:cv_critical_points}). 

The Hessian of the continuous energy at $\phi^*$ is given by
$$
E''(\uphi)\varphi = H \varphi + |\uphi|^2 \varphi + 2 \uphi \Re(\uphi\overline{\varphi}),
$$
for $\varphi \in H^1_0(\mathcal{D})$, where $H = -\Delta + V -\Omega L$. Note that, from $\|H \varphi\|_{H^{-1}} \leq C\|\varphi\|_{H^1}$ and usual Sobolev injections, we find that $E''(\phi^*)$ is a continuous linear map from $H^{1}_0(\mathcal{D})$ to $H^{-1}(\mathcal{D})$. As a slight abuse of notations, throughout the section and depending on the context, the bracket $\langle \cdot, \cdot \rangle$ may represent either the usual $L^2$-product as defined in the introduction or the associated $H^{-1}$-$H^1_0$ duality pairing. For the remainder of the section, we define
$$
F = \enstq{\varphi \in H^1_0(\mathcal{D})}{\langle \varphi,\uphi\rangle=\langle \varphi, i\uphi \rangle =0}.
$$ 
Similarly, we define $F_n = \enstq{v \in X(\mathcal{M}_n)}{\langle v, u_n^*\rangle = \langle v, iu_n^*\rangle = 0}$, that is, $F_n = T_{u_n^*}^\C \mathbb{S}_{h_n}$ when using the notations from the previous section. 
\begin{proposition}\label{th:spectral_gap_convergence}
  Let $(\lambda_n)_{n\in \N}$ and $\lambda$ be defined by 
  \begin{align*}
    \lambda_n = \inf_{v \in F_n \backslash \{0\}} \frac{\left\langle (E''_{h_n}(\uu_n) - \mu_n \Id)v, v\right\rangle}{\|v\|_{L^2}^2} \quad \text{and} \quad 
    \lambda = \inf_{\varphi \in F \backslash \{0\}} \frac{\left\langle (E''(\uphi) - \mu \Id)\varphi, \varphi \right\rangle}{\|\varphi\|_{L^2}^2}.
  \end{align*}
 Then $\lim_{n \rightarrow +\infty} \lambda_n = \lambda$.
\end{proposition}

As a direct consequence of Proposition \ref{th:spectral_gap_convergence}, if $\uphi$ is a quasi-isolated ground state (which corresponds to a definition analogous to Assumption \ref{ass:quasi_coercivity} in the continuous case, see also \cite{HenningYadavDiscrete2025}), then $\lambda_n$ is bounded from below by a positive constant for $n$ large enough, ensuring uniformity in the exponential convergence rates obtained in Theorem~\ref{th:exp_convergence}. We also note that, as was noted throughout the proofs of Section~\ref{sec:convergence_gradient_flow}, the constants $C$ (\emph{resp.} $\delta$) in the statement of Theorem~\ref{th:exp_convergence} may be bounded by $|u_n^*|_{1, 2}$ (\emph{resp.} $|u_n^*|_{1, 2}$ and $\varepsilon$). Therefore, since $|u_n^*|_{1,2}$ is bounded from above, all constants appearing in Theorem~\ref{th:exp_convergence} may be bounded uniformly in $n$ for the sequence of discrete minimizers $(u_n^*)_{n \in \N}$. On the other side, as shown in the recent work \cite{FengHenningTang2026}, if the potential $V$ and the domain are radially symmetric, it holds $\lambda = 0$ if $\uphi$ is non-radial. Consequently, one could expect local convergence rates to become degenerate as the mesh diameter $h$ goes to zero. Such behavior is exhibited in the numerical experiments of Section \ref{sec:numerics}.

In order to prove Proposition~\ref{th:spectral_gap_convergence}, we define the bilinear forms
\begin{equation}\label{eq:eigenproblem_discrete}
  a_n(v, w) = \langle E_{h_n}''(\uu_n) v, w\rangle
\end{equation}
for $v, w \in X(\mathcal{M}_n)$, and 
\begin{equation}\label{eq:eigenproblem_continuous}
  a(\varphi, \chi) = \langle E''(\uphi) \varphi, \chi\rangle
\end{equation}
for $\varphi, \chi \in H^1_0(\mathcal{D})$. We prove the following lemma, giving the continuity and coercivity of these bilinear forms.
\begin{lemma}\label{lemma:continuity_coercivity_bilinear}
  There exists $C > 0$ such that, for all $\varphi \in H^1_0(\mathcal{D})$, we have 
  $$
  a(\varphi, \varphi) \geq C \|\varphi\|_{H^1}^2.
  $$
  Moreover, there exists $C_1, C_2 > 0$ such that, for all $n \in \N$ and $v \in X(\mathcal{M}_n)$, we have 
  $$
  C_1 |v|_{1, 2}^2 \geq a_n(v, v) \geq C_2 |v|_{1, 2}^2.
  $$
\end{lemma}

\begin{proof}
  We begin with the first claim. For $\varphi \in H^1_0(\mathcal{D})$, we have 
  $$
    \langle E''(\phi^*)\varphi, \varphi \rangle = \langle H \varphi, \varphi\rangle + \langle |\phi^*|^2 \varphi, \varphi\rangle + 2\langle \phi^* \Re(\phi^*\overline{\varphi}), \varphi\rangle,
  $$
  Then, using using 
  $$
  \langle \phi^* \Re(\phi^*\overline{\varphi}), \varphi \rangle = \Re\left(\int_{\mathcal{D}} \phi^* \Re(\phi^* \overline{\varphi}) \overline{\varphi}\right) = \int_\mathcal{D} \Re(\phi^* \overline{\varphi})^2,
  $$
  we obtain $\langle E''(\phi^*) \varphi, \varphi \rangle \geq \langle H \varphi, \varphi \rangle$. The coercivity of $a$ then follows from the coercivity of $H$ using the assumption \eqref{eq:assumption_V} on the potential (see \cite[Appendix A]{DodingHenning2024} for the proof).

  For the second claim, a similar computation shows that $\langle E_{h_n}''(u_n^*) v, v\rangle \geq \langle H_{h_n} v, v\rangle$, so that $a_n(v, v) \geq C|v|_{1, 2}^2$ may be deduced from Proposition~\ref{pr:equivalence_h1}. The last inequality may be obtained from $\langle H_{h_n} v, v\rangle \leq C |v|_{1, 2}^2$ given in Proposition~\ref{pr:equivalence_h1}, as the sequence of meshes is assumed to be uniformly regular, as well as discrete Sobolev inequalities combined with the fact that $|u_n^*|_{1,2}$ is bounded independently of $n$.
\end{proof}

From Lemma~\ref{lemma:continuity_coercivity_bilinear} and using the Lax--Milgram theorem, for all $f \in L^2(\mathcal{D})$, there exists a unique $\varphi \in F$ such that, for all $\chi \in F$, we have $a(\varphi, \chi) = \langle f, \chi\rangle$. Using the compact injection of $H^1_0(\mathcal{D})$ in $L^2(\mathcal{D})$ and defining the operator $T$ by $Tf = \varphi$, we find that $T$ is a self-adjoint compact operator on $L^2(\mathcal{D})$. Similarly, we may define the compact $L^2(\mathcal{D})$ operator $T_n$ by $a(T_n f, v) = \langle f, v\rangle$ for all $v \in F_n$. One also verifies easily the identities $\ker(T) = F^\perp$ and $\ker(T_n) = F_n^\perp$ where the orthogonal complement is taken over $L^2(\mathcal{D})$. In particular, $\ker(T)$ is generated by $\phi^*$ and $i\phi^*$ while $\ker(T_n)$ is finite codimensional of codimension $\dim(F_n) = \dim(X(\mathcal{M}_n)) - 2 $. Therefore, $0$ is an eigenvalue of $T$ of multiplicity 2 while it is of infinite multiplicity for $T_n$.

We may now define the eigenbases $(\alpha^k_n, v^k_n)_{k \in \N^*}$ and $(\alpha^k, \varphi^k)_{k \in \N^*}$ of $T_n$ and $T$ respectively, where all eigenvalues are nonnegative (from Lemma~\ref{lemma:continuity_coercivity_bilinear}) and arranged in non-increasing order. For all nonzero eigenvalues, we may define $\beta_n^k = 1/\alpha_n^k$ and $\beta^k = 1/\alpha^k$ so that the eigenpairs $(\beta_n^k, v_n^k)_{k \in \{1, \dots, \dim(F_n)\}}$ and $(\beta^k, \varphi^k)_{k \in \N^*}$ are solutions to the constrained eigenvalue problems 
\begin{equation}\label{eq:def_beta_k_n}
a_n(v_n^k, w) = \beta_n^k \langle v_n^k, w\rangle,
\end{equation}
for all $w \in F_n$ and 
\begin{equation}\label{eq:def_beta_k}
a(\varphi^k, \chi) = \beta^k \langle \varphi^k, \chi\rangle,
\end{equation}
for all $\chi \in F$. From these considerations, one readily proves the following.
\begin{lemma}\label{lemma:expressions_lambda}
  Let $\lambda$ and $(\lambda_n)_{n \in \N}$ be defined as in Proposition~\ref{th:spectral_gap_convergence}. We have
  $$
  \lambda = \frac{1}{\alpha^1} - \mu \text{ and } \lambda_n = \frac{1}{\alpha_n^1} - \mu_n,
  $$
  for all $n \in \N$, where $\alpha^1$ and $\alpha_n^1$ are defined as above.
\end{lemma}

In order to prove Proposition~\ref{th:spectral_gap_convergence}, we may therefore show the convergence of eigenvalues of $T_n$ towards eigenvalues of $T$. Using Babuska--Osborn theory, this is equivalent to showing the convergence of $T_n$ towards $T$ in operator norm \cite{BabuskaOsbornEigenvalue1991}. More specifically, we adapt the arguments from \cite{KopruckiEymardConvergence2007}, in which spectral convergence problems are considered for finite volume schemes, to the \emph{constrained} eigenvalue problems of interest in our context. We begin by proving the following lemma, giving properties referred to as compactness and consistency in \cite{KopruckiEymardConvergence2007}.

\begin{lemma}\label{lemma:compactness_consistency_hessian}
  Let $(v_n)_{n \in \N}$ be a sequence in $X(\mathcal{M}_n)$ such that there exists $C \geq 0$ with $a_n(v_n, v_n) \leq C$ for all $n \in \N$. Then, up to a subsequence, $(v_n)_n$ converges in $L^2(\mathcal{D})$ towards a limit $\varphi \in H^1_0(\mathcal{D})$. Furthermore, for such a subsequence and for all $\chi \in \mathcal{C}^\infty_c(\mathcal{D})$, we have 
  $$
  \lim_{n \rightarrow \infty} a_n(v_n, P_{h_n}\chi) = a(\varphi, \chi).
  $$
\end{lemma}

\begin{proof}
  We begin by proving the existence of a limit for $(v_n)_{n \in \N}$. From Lemma~\ref{lemma:continuity_coercivity_bilinear}, we know that $a_n(v_n, v_n) \leq C$ implies that $(|v_n|_{1,2})_{n \in \N}$ is bounded. Therefore, $(v_n)_{n \in \N}$ converges up to a subsequence in $L^2(\mathcal{D})$ to $\varphi \in H^1_0(\mathcal{D})$. Moreover, we know from Proposition \ref{pr:compactness_consistency} that for all $\chi \in \mathcal{C}^\infty_c(\mathcal{D})$,  
  $$
  \lim_{n \rightarrow \infty} \langle H_{h_n} v_n, P_{h_n} \chi\rangle = \langle H \varphi, \chi \rangle.
  $$
  It thus remains to show that $\langle |\uu_n|^2 v_n, v_n\rangle$ and $\langle \uu_n \Re(\uu_n \overline{v_n}), v_n\rangle$ converge to $\langle |\uphi|^2 \varphi, \varphi\rangle$ and $\langle \uphi \Re(\uphi\overline{\varphi}), \varphi\rangle$ respectively. For this, we recall that $(\uu_n)_{n \in \N}$ and $(v_n)_{n \in \N}$ converge respectively to $\uphi$ and $\varphi$ in $L^2(\mathcal{D})$ and that $|\uu_n|_{1, 2}$ and $|v_n|_{1, 2}$ are uniformly bounded. Therefore, we may proceed with the same argument used in the proof of Proposition \ref{pr:cv_critical_points} to show that convergence occurs in $L^p(\mathcal{D})$ for all $p < 6$ to conclude.
\end{proof}

We are now able to prove the convergence of $T_n$ towards $T$.

\begin{lemma}\label{lemma:convergence_resolvants}
  We have 
  $$
  \lim_{n \rightarrow \infty} \|T_n - T\|_{\mathcal{L}(L^2(\mathcal{D}))} = 0,
  $$
  where $\|\cdot\|_{\mathcal{L}(L^2(\mathcal{D}))}$ is the operator norm on continuous linear maps from $L^2(\mathcal{D})$ to itself.
\end{lemma}

\begin{proof}
We adapt the proof from \cite{KopruckiEymardConvergence2007}. Lemma~\ref{lemma:convergence_resolvants} being false is equivalent to the existence of both a sequence $(f_n)_{n \in \N}$ normalized in $L^2(\mathcal{D})$ and some $\varepsilon > 0$ such that
  $$
  \|T_n f_n - T f_n\|_{L^2} \geq \varepsilon,
  $$
  for all $n$ in $\N$. Therefore, one only needs to pick such a sequence arbitrarily and prove that $\|T_n f_n - T f_n\|_{L^2} \rightarrow 0$. Defining $v_n = T_n f_n$, we have $a_n(v_n, w) = \langle f_n, w\rangle$
  for all $w$ in $F_n$. In particular, since $v_n \in F_n$ and $\|f_n\|_{L^2}=1$, we have
  \begin{equation}\label{eq:first_bound_a_n_v_n}
  a_n(v_n, v_n) \leq \|v_n\|_{L^2}.
  \end{equation}
  Moreover, using the coercivity of $a_n$ from Lemma~\ref{lemma:continuity_coercivity_bilinear} and the discrete Poincaré inequality, we obtain $a_n(v_n, v_n) \geq C |v_n|_{1, 2}^2 \geq C \|v_n\|_{L^2}^2$. Combining this inequality together with \eqref{eq:first_bound_a_n_v_n}, we obtain
  $$
  C\|v_n\|_{L^2}a_n(v_n, v_n)^{1/2} \leq a_n(v_n, v_n) \leq \|v_n\|_{L^2},
  $$
  so that $(a_n(v_n, v_n))_{n \in \N}$ is bounded. From Lemma \ref{lemma:compactness_consistency_hessian}, there thus exists $\varphi \in H^1_0(\mathcal{D})$ such that $(v_n)_{n\in \N}$ converges up to a subsequence to $\varphi$ in $L^2(\mathcal{D})$. Moreover, since $(f_n)_{n \in \N}$ is bounded in $L^2(\mathcal{D})$, we may extract another subsequence so that $(f_n)_{n \in \N}$ converges weakly to a limit $f \in L^2(\mathcal{D})$.

  We now wish to show that $\varphi = Tf$. First, since $v_n \in F_n$, we have $\langle v_n, \uu_n\rangle = \langle v_n, i \uu_n\rangle = 0$. Using that $\uu_n$ converges to $\uphi$ in $L^2(\mathcal{D})$ and passing to the limit, we find that $\varphi\in F$. It remains to show that, for all $\chi \in F$, we have 
  $$
  a(\varphi, \chi) = \langle f, \chi\rangle.
  $$
  We now define the projections $\Pi$ and $\Pi_n$ on $L^2$ by 
  $$
  \Pi g = g - \langle g, \phi^* \rangle \phi^* - \langle g, i\phi^* \rangle i \phi^* \text{ and } \Pi_n g = g - \langle g, u_n^* \rangle u_n^* - \langle g, iu_n^* \rangle i u_n^*,
  $$
  for $g \in L^2(\mathcal{D})$ and $n \in \N$. Note that, in general, $\Pi_n g$ is not in $F_n$, but that $\Pi_n$ does map $X(\mathcal{M}_n)$ to $F_n$. From the convergence of $u_n^*$ to $\phi^*$ in $L^2(\mathcal{D})$, it is then clear that $\Pi_n$ converges to $\Pi$ in $\mathcal{L}(L^2(\mathcal{D}))$. Therefore, since $v_n = T_n f_n$, we have for all $\chi \in \mathcal{C}^\infty_c(\mathcal{D}) \cap F$ that
  \begin{equation}\label{eq:relation_v_n}
    a_n(v_n, \Pi_n P_{h_n} \chi) = \langle f_n, \Pi_n P_{h_n} \chi \rangle,
  \end{equation}
  as $\Pi_n P_{h_n} \chi$ is in $F_n$. Since $f_n$ converges weakly to $f$, $P_{h_n} \chi$ converges strongly to $\chi$ using Proposition~\ref{pr:approx_interpolants} and $\Pi_n$ converges to $\Pi$, we find that the right-hand side of \eqref{eq:relation_v_n} converges to $\langle f,  \Pi \chi\rangle = \langle f, \chi \rangle$. For the left-hand side of \eqref{eq:relation_v_n}, we write 
  $$
  a_n(v_n, \Pi_n P_{h_n} \chi) = a_n(v_n, P_{h_n} \chi) - a_n(v_n, \Pi_n^\perp P_{h_n} \chi),
  $$
  where $\Pi_n^\perp \coloneq \Id - \Pi_n$. From Lemma \ref{lemma:compactness_consistency_hessian}, we know that $a_n(v_n, P_{h_n}\chi)$ converges to $a(\varphi, \chi)$. Moreover, from Lemma~\ref{lemma:continuity_coercivity_bilinear} and the fact that $a_n$ is symmetric, we find
  $$
  |a_n(v_n, \Pi_n^\perp P_{h_n} \chi)| \leq C|v_n|_{1, 2} |\Pi_n^\perp P_{h_n} \chi|_{1, 2} \leq C|\Pi_n^\perp P_{h_n} \chi|_{1, 2},
  $$
  where we used that $|v_n|_{1, 2}$ is bounded. It thus remains to show that $|\Pi_n^\perp P_{h_n} \chi|_{1, 2}$ converges to zero. For $g \in L^2(\mathcal{D})$, we have the expression $\Pi_n^\perp g = \langle g, u_n^* \rangle \uu_n + \langle g, iu_n^* \rangle i\uu_n$, so that 
  $$
  |\Pi_n^\perp P_{h_n} \chi|_{1, 2} \leq \left(|\langle \uu_n, P_{h_n} \chi\rangle| + |\langle i \uu_n, P_{h_n} \chi\rangle|\right)|\uu_n|_{1, 2}.
  $$
  Since $\uu_n$ converges to $\uphi$ in $L^2$, $P_{h_n} \chi$ converges to $\chi$ in $L^2$ using Proposition~\ref{pr:approx_interpolants}, $\chi$ is in $F$, and $|\uu_n|_{1, 2}$ is bounded, we find the desired result. As a consequence, we find that $\varphi \in F$ and satisfies 
  $$
  a(\varphi, \chi) = \langle f, \chi \rangle,
  $$
  for all $\chi \in \mathcal{C}^\infty_c(\mathcal{D}) \cap F$. By density of $\mathcal{C}^\infty_c(\mathcal{D})$ in $H^1_0(\mathcal{D})$, we find that this identity is true for all $\chi \in F$ so that $\varphi = Tf$.

  We now conclude the proof. By convergence of $(v_n)_{n \in \N}=(T_nf_n)_{n\in\N}$ towards $\varphi=Tf$ in $L^2(\mathcal{D})$, we find that 
  $$
  \lim_{n \rightarrow \infty} \|T_n f_n - T f\|_{L^2} = 0.
  $$
  Moreover, as $(f_n)_{n \in \N}$ converges weakly in $L^2(\mathcal{D})$ towards $f$ and since $T$ is compact as noted earlier in the section, we find that $T f_n$ converges strongly in $L^2(\mathcal{D})$ towards $T f$. As such, using that
  $$
  \|T_n f_n - T f_n\|_{L^2} \leq \|T_n f_n - T f\|_{L^2} + \|T f - T f_n\|_{L^2},
  $$
  we find that $\|T_n f_n - T f_n\|_{L^2} \rightarrow 0$ as claimed.
\end{proof}

\begin{proof}[Proof of Proposition \ref{th:spectral_gap_convergence}]
  We recall that $T_n$ converges to $T$ in operator norm and that $(\alpha_n^k)_{k \in \N^*}$ and $(\alpha^k)_{k \in \N^*}$ are the eigenvalues of $T_n$ and $T$ respectively, arranged in non-increasing order. Let $m$ be the multiplicity of the first eigenvalue of $T$, that is, 
  $$
  \alpha^1 = \dots = \alpha^m > \alpha^{m+1}.
  $$
  Note that $m$ is necessarily finite from the compactness of $T$. As a consequence, there exists $\varepsilon > 0$ such that $I \coloneq (\alpha^1 - \varepsilon, +\infty)$ separates the $m$ first eigenvalues of $T$ from the rest of the spectrum, that is $I \cap \sigma(T) = \{\alpha^1, \dots, \alpha^m\} = \{\alpha^1\}$.
  
  Then, \cite[Theorem 2.1]{KopruckiEymardConvergence2007} gives the existence of $N \in \N$ such that, for all $n \geq N$, we have $I \cap \sigma(T_n) = \{\tilde{\alpha}_n^1, \dots, \tilde{\alpha}_n^m\}$ and $\lim_{n \to \infty} \tilde{\alpha}_n^j = \alpha^1$ for all $j \in \{1, \dots, m\}$. Therefore, up to a rearrangement of the eigenvalues, we necessarily have $\tilde{\alpha}_n^j = \alpha_n^j$. In particular, this implies that $\lim_{n \to +\infty} \alpha_n^1 = \alpha^1$.

  Recalling that $\lambda_n = 1/\alpha_n^1 - \mu_n$ and $\lambda = 1/\alpha^1 - \mu$ and that $\mu_n$ converges to $\mu$, we are able to conclude.
\end{proof}

\section{Numerical results} \label{sec:numerics}

\subsection{Numerical implementation}
Numerical schemes described in this section are implemented in Python. Two-dimensional triangular meshes are generated using the frontal Delaunay method of the Gmsh library \cite{GeuzaineGMSH2009} and then converted into a custom mesh structure in order to compute the necessary data for finite volume schemes, such as mesh connectivity and triangle circumcenters. We choose to mesh a disk of radius $R$ centered at the origin. Note that, compared to the theoretical analysis carried out above, the domain $\mathcal{D}$ is a polygonal approximation of the disk which depends on the mesh itself. Also, we introduce the parameters $\gamma, m > 0$, which were not present in the analysis for convenience, rewriting
\[ H_h = -\frac{1}{2m}A_h + V_h - \Omega L_h \quad \text{and} \quad E_h(u) = \langle H_h u, u \rangle + \frac{\gamma}{2}\|u\|_{L^4}^4. \]
We consider a quadratic potential, parametrized by
$$
V(x) = V_0\left((1+\varepsilon)x_1^2 + (1-\varepsilon)x_2^2\right)
$$
where $V_0 > 0$ controls the strength of the potential and $\varepsilon \in (-1, 1)$ the ellipticity. Note that $\varepsilon=0$ corresponds to an axisymmetric potential.

In practice, one also needs a time discretization of \eqref{eq:gradient_flow} in order to compute minimizers of the discrete energy. In our case, we implemented a standard semi-implicit normalized gradient descent scheme, also called imaginary time evolution method in the physics literature, which is often used for computing Gross--Pitaevskii ground states, see e.g. \cite{BaoDu2004,FaouJezequel2018,HenningDependency2023}. Such scheme is given by
\begin{equation}\label{eq:normalized_gradient_descent}
    \left\{ \begin{aligned}
        & \frac{\tilde{u}_{n+1}-u_n}{\tau} = -H_h \tilde{u}_{n+1} - |u_{n}|^2 \tilde{u}_{n+1},
        \\
        & u_{n+1} = \frac{\tilde{u}_{n+1}}{\|\tilde{u}_{n+1}\|_{L^2}},
    \end{aligned} \right.
\end{equation}
where $\tau > 0$ is a time-step parameter. We refer to \cite{BaoDu2004,BaoCai2013} for discussions on the relations between \eqref{eq:gradient_flow} and \eqref{eq:normalized_gradient_descent}. Although we have chosen not to study exponential convergence towards minimizers for \eqref{eq:normalized_gradient_descent} in this article for brevity, we expect that the tools used in \cite{FaouJezequel2018} or \cite{HenningDependency2023} could be repurposed in our context, adapting the computations of Section \ref{sec:gradient_flow}.

All codes used throughout this section are available at the \textsc{Gitlab} page \url{https://plmlab.math.cnrs.fr/rroussel/fv-gross-pitaevskii}. Note that the repository also contains all the files and data in order to reproduce the upcoming figures.

\subsection{Consistency of the discretization}

We begin by studying the numerical consistency of the finite volume discretization studied throughout this paper. We introduce the following function $\varphi$, which will be used for the numerical experiments of this section 
\begin{equation}\label{eq:explicit_u}
\varphi(r,\theta) = \sqrt{\frac{12}{\pi}}\frac{1}{R^4} re^{i\theta}(R^2-r^2),
\end{equation}
where $(r, \theta)$ denotes the usual polar coordinates. One then verifies that $\varphi$ is smooth, vanishes at the boundary and is normalized in $L^2(\mathcal{D})$. Moreover, one computes explicitly that $L \varphi = \varphi$ and that its energy is given by
\begin{equation}\label{eq:explicit_energy}
E(\varphi) = \frac{4}{mR^2} - \Omega + \frac{1}{5}V_0 R^2 + \frac{12}{35}\frac{\gamma}{\pi R^2},
\end{equation}
in the simpler case $\varepsilon=0$. In Figure~\ref{fig:consistency}, we consider two quantities as functions of the mesh parameter $h$, with parameters $R=4$, $m=1$, $\Omega=1.8$, $V_0=5$ and $\gamma=100$. First, we consider the $L^2$-error on the angular momentum $L \varphi$ compared to our finite volume discretization. More specifically, we compute for simplicity the $L^2$-norm of the error interpolated on the centers of the cells, that is, $L_h P_h \varphi - P_h L \varphi$. Moreover, we consider the global error between the discrete energy and the continuous one, given explicitly by \eqref{eq:explicit_energy}. The computations are performed for meshes having a number of triangles ranging from 507 to 119,461.
\begin{figure}[ht]
  \centering 
  \includegraphics{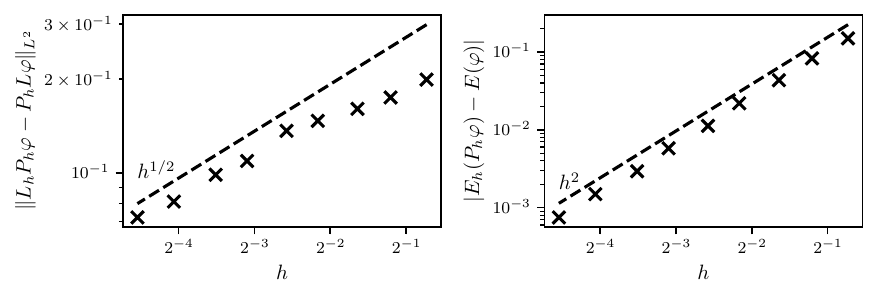}
  \caption{$L^2$-error of the angular momentum \textit{(left)} and energy error \textit{(right)} for $\varphi$ given by \eqref{eq:explicit_u}.}
  \label{fig:consistency}
\end{figure}

We observe in Figure~\ref{fig:consistency} that the $L^2$-error in angular momentum is of order $h^{1/2}$ while the energy error is of order $h^2$. We note that the latter exceeds the linear error in energy for $H^2$ functions given by Proposition~\ref{pr:convergence_energy}.

\subsection{Convergence of the normalized gradient descent algorithm}

In this section, we study the numerical convergence of \eqref{eq:normalized_gradient_descent} towards minimizers of the discrete energy. The computations are performed on a fixed mesh with 30,025 triangles and $h \approx 0.088$. We once again choose the parameters $R=4$, $m=1$, $\Omega = 1.8$, $V_0=5$ and $\gamma=100$, as well as the time step $\tau=1$. We initialize the gradient descent with the function
$$
u_0(r, \theta) = C r^\ell e^{i\ell\theta}(R^2-r^2),
$$
interpolated on the mesh with $\ell =2$. We note that, for $\Omega$ large enough, the final solution $u^*$ may be sensitive to the initial condition and that convergence towards a global minimizer may not be expected in general. Moreover, the number of iterations needed to have satisfying convergence also depends greatly on the initial condition. 

We consider the two cases, $\varepsilon=0$ and $\varepsilon=0.2$. In the rotational symmetric case $\varepsilon=0$, we know from Section \ref{sec:spectral_gap} that for non-radial minimizers, the theoretical convergence rate $\lambda$ given by \eqref{eq:lambda} becomes degenerate for finer meshes. This phenomenon may also be observed numerically, as illustrated in Figure \ref{fig:spectral_gap}. In the asymmetric case $\varepsilon=0.2$, Assumption~\ref{ass:quasi_coercivity} is expected to hold in the continuous case. Therefore, as shown in Section~\ref{sec:spectral_gap}, Assumption~\ref{ass:quasi_coercivity} is also verified in the discrete case, with convergence rates independent of the mesh size $h$, as observed in Figure~\ref{fig:spectral_gap}. We note here that $\lambda_h$ is simply computed as the difference between the first two eigenvalues of the Hessian. The authors also implemented a method to solve the constrained eigenvalue problems \eqref{eq:def_beta_k_n}, but the results were unstable, longer to obtain and gave almost identical eigenvalues as for the simpler unconstrained problems. 
\begin{figure}[ht]
  \centering
  \includegraphics{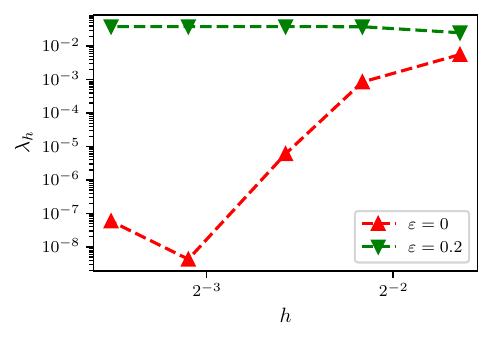}
  \caption{Spectral gap \eqref{eq:lambda} as a function of $h$ in the cases $\varepsilon=0$ and $\varepsilon=0.2$.}
  \label{fig:spectral_gap}
\end{figure}

In Figure~\ref{fig:convergence_gradient_descent}, we plot the convergence of solutions to \eqref{eq:normalized_gradient_descent} towards a minimizer in both cases $\varepsilon=0$ and $\varepsilon=0.2$. Such minimizer $u^*$ is computed after 10,000 iterations and is compared with the first 8,000 steps. In the case $\varepsilon=0$, after an initial fast decay of the error, the gradient descent \eqref{eq:normalized_gradient_descent} enters a regime where the convergence rate is drastically reduced. The opposite is observed in the case $\varepsilon=0.2$, for which the dynamics starts with a plateau for a few hundred iterations before reaching a steady exponential convergence rate. As a consequence, the method is much more efficient in nonsymmetric configurations to compute minimizers with multiple vortices, reaching machine precision in energy within approximately 4,000 iterations. This suggests that rotational symmetry should be taken into account in the numerical approach.

\begin{figure}[ht]
    \centering
    \includegraphics{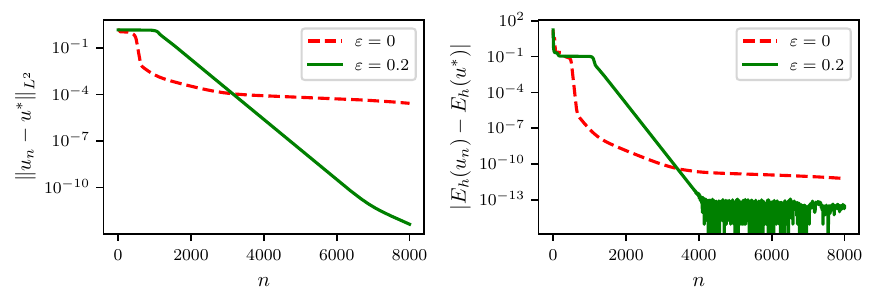}
    \caption{$L^2$-error \textit{(left)} and energy error \textit{(right)} for the normalized gradient descent in the cases $\varepsilon=0$ and $\varepsilon=0.2$.}
    \label{fig:convergence_gradient_descent}
\end{figure}

\subsection{Convergence of discrete minimizers towards continuous ground states}
We now consider the convergence of minimizers $u_h^*$ as a function of $h$. The physical parameters are the same as in the previous sections, and we focus on the non-symmetric case $\varepsilon=0.2$. The normalized gradient descent \eqref{eq:normalized_gradient_descent} was computed for 10,000 steps on meshes containing between 1,919 and 30,025 triangles, with mesh size $h$ from $h\approx 0.17$ to $h \approx 0.088$ respectively. These minimizers are compared with a reference solution $\phi^*$, also computed after 10,000 steps of \eqref{eq:normalized_gradient_descent} on a finer mesh with 119,461 triangles and $h \approx 0.043$. A plot of this reference minimizer is shown in Figure~\ref{fig:minimizers}.

\begin{figure}
  \centering 
  \includegraphics[width=\textwidth]{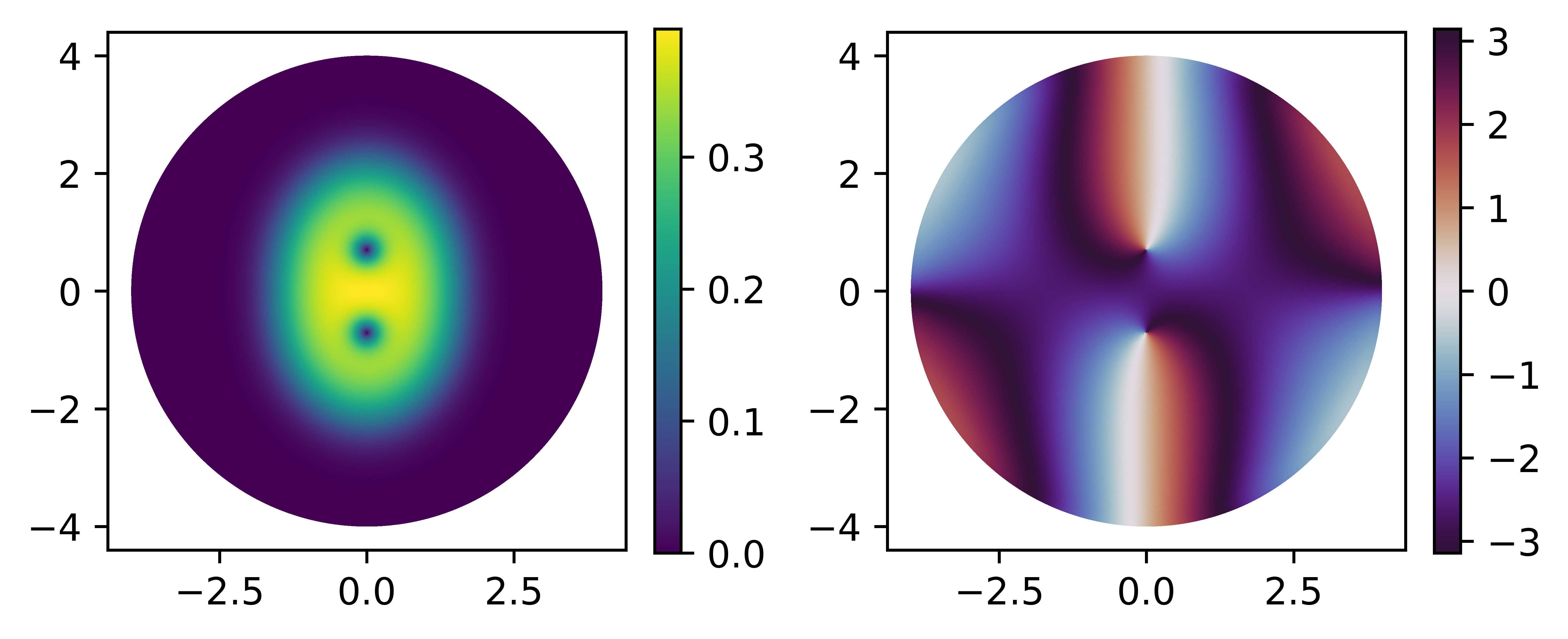}
  \caption{Modulus \textit{(left)} and argument \textit{(right)} of the two-vortices reference solution obtained on the fine mesh.}
  \label{fig:minimizers}
\end{figure}

In Figure~\ref{fig:convergence_minimizers}, we plot the $L^2$-error and the energy error as functions of $h$. In order to compute the $L^2$-error, we note two technicalities which are encountered. First, one needs to compute the $L^2$-norm of a difference of two functions which are piecewise constant on two different meshes. Numerically, this can be obtained by computing the area of all the intersections of the triangles of the two meshes. Second, one may only expect convergence of the minimizers up to a phase. Therefore, the global phases of the coarse solutions $u_h^*$ are chosen in order to minimize the $L^2$-error with the reference solution $\phi^*$. This can be found explicitly by computing the two critical points of this error as a function of the phase, and by choosing the minimum among the two. 

We notice a linear convergence of the $L^2$-error as a function of $h$, as expected by the optimal $L^2$-approximation rate provided by Proposition~\ref{pr:approx_interpolants}, as well as a quadratic convergence in the energy.

\begin{figure}[ht]
    \centering
    \includegraphics{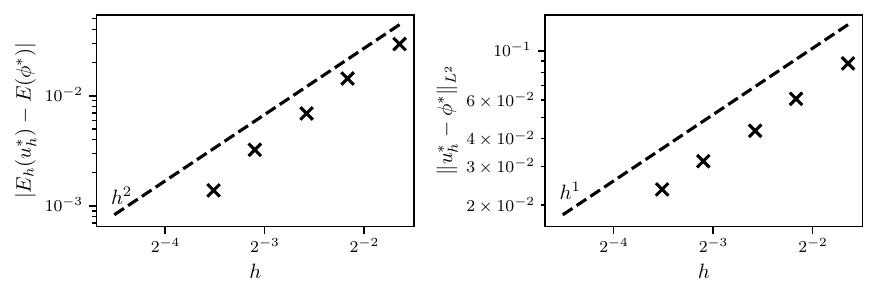}
    \caption{Energy \textit{(left)} and $L^2$-error \textit{(right)} compared to the reference solution $\phi^*$ as functions of $h$.}
    \label{fig:convergence_minimizers}
\end{figure}

\subsection{Ground states in exotic geometries} Recent BEC experiments allow for increasingly exotic geometric configurations requiring efficient numerical methods on general geometries. We illustrate the flexibility of the numerical method developed in this article by computing the ground state of moon-shaped BEC, which has the same qualitative behavior as a shell-shaped BEC under the influence of gravity, see \cite[Figure 3]{Rhyno2026}. More precisely, we consider a potential of the form
\[  V= V_0r^2+V_1 \exp\left(-\frac{(x_1-0.5)^2+x_2^2)}{2\sigma} \right),\]
which is the combination of a standard harmonic trap with an off-set repulsive Gaussian \cite{Wen2011}. We take numerical parameters as follows: $\gamma=100$, $m=1$, $V_0=50$, $V_1=1000$, $\sigma=0.2$, and $\Omega=8$. We compute in Figure \ref{fig:moonshape} discrete ground states associated with the discretization of three different domains $\mathcal{D}$ of same mesh size $h\simeq0.04$: a square $\left[-2,2\right]^2$ (composed of 41656 triangles), a centered disk of radius $2$ (32447 triangles) and a moon-shaped domain (21599 triangles). Qualitatively, the ground states look identical, which demonstrates the efficiency of the moon-shaped geometry, meshing only in areas of physical interest.

\begin{figure}
  \centering 
  \includegraphics[width=\textwidth]{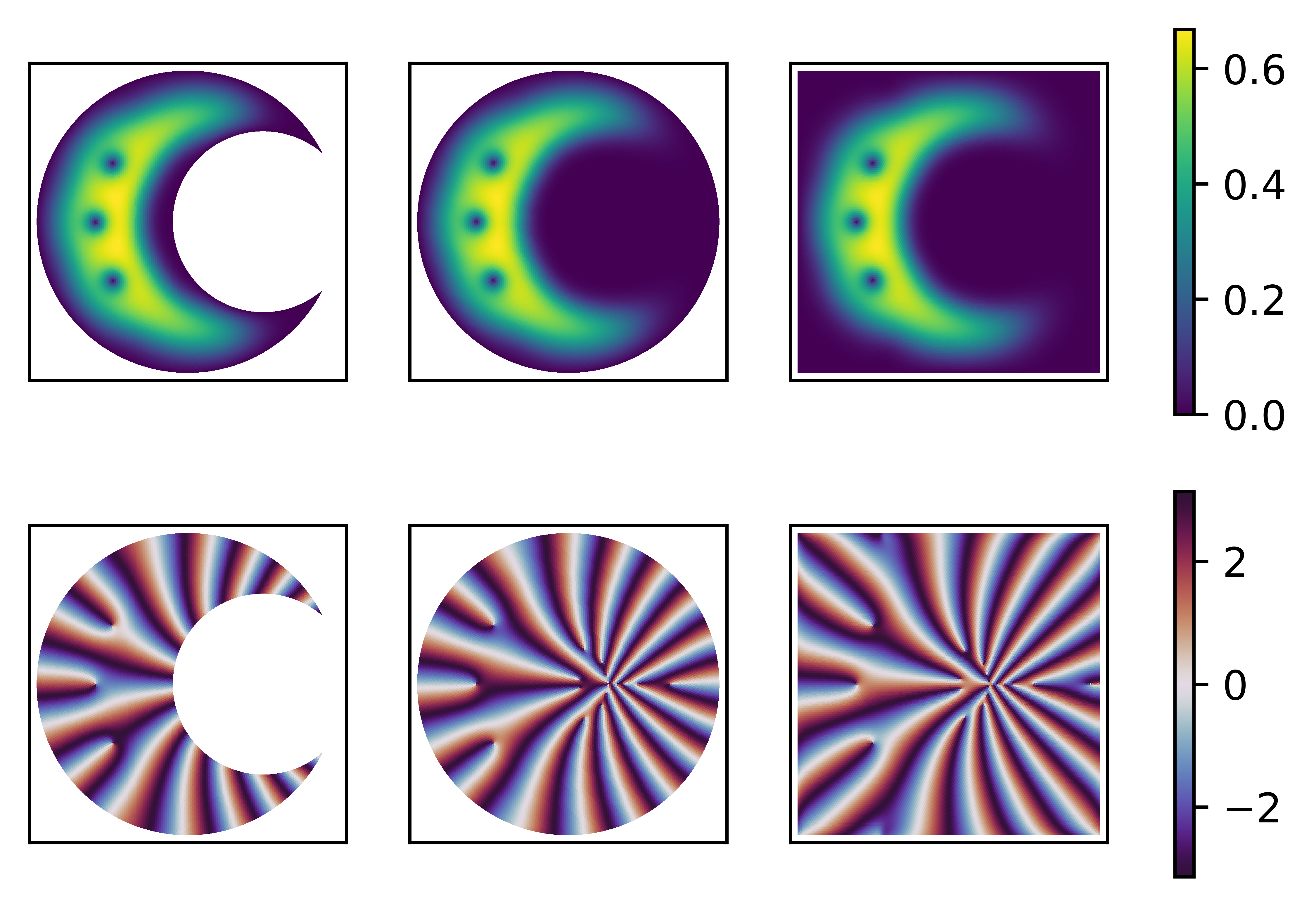}
  \caption{Modulus (\emph{top}) and argument (\emph{bottom}) of discrete ground states on various meshes, displaying 3 vortices on a moon-shaped BEC.}
  \label{fig:moonshape}
\end{figure}

\subsection*{Acknowledgements}

The authors would like to thank Flore Nabet for her insights on constrained eigenvalue problems in the finite volume framework. The authors also acknowledge the support of the CDP C2EMPI, together with the French State under the France-2030 programme, the University of Lille, the Initiative of Excellence of the University of Lille, the European Metropolis of Lille for their funding and support of the R-CDP-24-004-C2EMPI project.


\bibliographystyle{siam}
\bibliography{biblio}

\end{document}